\documentclass[12pt,twoside]{amsart}
\usepackage{amssymb}
\usepackage{verbatim}
\usepackage{amsmath}
\usepackage{bm}
\usepackage{a4wide}
\usepackage[latin1]{inputenc}
\usepackage[T1]{fontenc}
\usepackage{times}
\usepackage{latexsym}
\usepackage{enumerate}
\usepackage{xcolor}
\usepackage[
  colorlinks=true,
  linkcolor=blue,
  citecolor=blue,
  urlcolor=blue
]{hyperref}

\makeatletter
\newcommand{\sumprime}{\if@display\sideset{}{'}\sum%
            \else\sum'\fi}
\makeatother

\begin{document}

\numberwithin{equation}{section}

\newtheorem{theorem}{Theorem}[section]
\newtheorem{proposition}[theorem]{Proposition}
\newtheorem{conjecture}[theorem]{Conjecture}
\def\theconjecture{\unskip}
\newtheorem{corollary}[theorem]{Corollary}
\newtheorem{lemma}[theorem]{Lemma}
\newtheorem{observation}[theorem]{Observation}
\newtheorem{definition}{Definition}
\numberwithin{definition}{section} 
\theoremstyle{remark}
\newtheorem{remark}[theorem]{Remark}
\newtheorem{kl}{Key Lemma}
\def\thekl{\unskip}
\newtheorem{question}{Question}
\def\thequestion{\unskip}
\newtheorem{example}{Example}
\def\theexample{\unskip}
\newtheorem{problem}{Problem}

\thanks{The first author is partially supported by the DFG funded projects
SFB TRR 191 `Symplectic Structures in Geometry, Algebra and Dynamics'
(Project-ID 281071066\,--\,TRR 191), DFG Priority Program 2265 `Random
Geometric Systems' (Project-ID 422743078), and the ANR--DFG project
`QuasiDy\,--\,Quantization, Singularities, and Holomorphic Dynamics'
(Project-ID 490843120).}
\thanks{The second author is supported by the National Natural Science Foundation of Sichuan, No. 2025ZNSFSC0796, the National Natural Science Foundation of China, No. 12401103, and the Fundamental Research Fund for the Central Universities, Southwest Minzu University, No. ZYN2025007.}

\title{Parameter Dependence of Weighted Bergman Kernels Beyond Smoothness}
 \author[George Marinescu]{George Marinescu}
 \author[Xu Xing]{Xu Xing}
  
\begin{abstract}
We study the parameter dependence of weighted Bergman kernels on fixed
bounded domains in $\mathbb C^n$. Our main result establishes
real-analytic dependence on $t\in(-1,\infty)$ for the kernels associated
with $\delta^t\,dV$ on bounded pseudoconvex domains with $C^2$ boundary,
where $\delta$ is the Euclidean distance to the boundary.
The parameter derivatives satisfy factorial estimates in
$C^\ell(S\times S)$ for every $S\Subset\Omega$ and $\ell\ge0$,
uniformly on compact parameter intervals.
The proof combines weighted $L^2$ estimates for $\bar\partial$ with a
holomorphic family of bounded operators and gives a local holomorphic
extension with values in a fixed weighted Bergman space.
For weight families smooth jointly in space and parameter up to the
boundary, we also express all parameter derivatives in terms of
iterated weighted Bergman projections and complete exponential Bell
polynomials. This formula implies preservation of the Gevrey class
$G^s$, $s\ge1$, under uniform parameter estimates up to the boundary.
We show that the same derivative formula holds for the weights
$-t\log\delta$.
 \bigskip


  \smallskip

\end{abstract}
\address[George Marinescu]{Department of Mathematics and Computer Science,
\newline\mbox{\quad}\,University of Cologne, Cologne, 50931, Germany
\newline\mbox{\quad}\,Institute of Mathematics `Simion Stoilow', 
Romanian Academy, Bucharest, Romania}
\email{gmarines@math.uni-koeln.de}
\address[Xu Xing] {School of Mathematics, Southwest Minzu University, Chengdu, 610041, China,
\newline\mbox{\quad}\,Department of Mathematics and Computer Science,
University of Cologne, Cologne, 50931, 
\newline\mbox{\quad}\,Germany
}
\email{xingxu@swun.edu.cn}

\maketitle

\section{Introduction}
Let $\Omega$ be a bounded domain in $\mathbb C^n$, let
$D\subset\mathbb R$ be open, and let $\varphi_t$ be a real-valued
measurable function on $\Omega$ for each $t\in D$.
Define
\[
A^2(\Omega,\varphi_t)
:=\left\{f\in\mathcal O(\Omega):
\|f\|_{\varphi_t}^2
:=\int_\Omega|f|^2e^{-\varphi_t}\,dV<\infty\right\},
\]
where $dV$ denotes Lebesgue measure. We call $\varphi_t$
admissible if the density $e^{-\varphi_t}$ is an admissible
weight, that is, $A^2(\Omega,\varphi_t)$ is closed in
$L^2(\Omega,\varphi_t)$ and every point evaluation is
continuous on it.
For an admissible weight, we denote the reproducing kernel and
the orthogonal Bergman projection by
$K_{\Omega,\varphi_t}$ and $P_{\Omega,\varphi_t}$,
respectively.

Our main result concerns $\varphi_t=-t\log\delta$, $t>-1$, where
$\delta=\delta_\Omega$ is the Euclidean distance to $\partial\Omega$.
On every bounded pseudoconvex domain with $C^2$ boundary, we prove
real-analytic dependence on $t$ by extending each kernel section
holomorphically in the parameter with values in a fixed weighted
Bergman space. Parameter Cauchy estimates, followed by interior
holomorphic estimates, give bounds locally uniform in the spatial
variables and for all spatial derivatives.
For the dependence of unweighted Bergman kernels on varying domains,
we refer to
\cite{Ramadanov67,Hamilton77,Hamilton79,GreeneKrantz82,Komatsu82,Chen16}.

For weights on a fixed domain, Pasternak-Winiarski proved
real-analytic dependence of the Bergman projection and reproducing
kernel under bounded $L^\infty$ perturbations of an admissible
weight \cite{PasternakWiniarski90}. He also obtained higher
differential and Taylor formulas for their dependence on the weight.
The boundary-distance family is not a bounded perturbation of any
fixed member of the family: if $t\ne t_0$, then
\[
\varphi_t-\varphi_{t_0}=-(t-t_0)\log\delta
\]
is unbounded near $\partial\Omega$. Equivalently, the quotient of
the densities is $\delta^{t-t_0}$, and either this quotient or its
reciprocal is unbounded near the boundary. Thus the densities
$\delta^t$ and $\delta^{t_0}$ are not uniformly comparable, so the
bounded-perturbation result of \cite{PasternakWiniarski90} does not
apply directly.
Without assuming parameter differentiability, H\"older estimates for diagonal Bergman kernels with plurisubharmonic weights were obtained under uniform H\"older control of $e^{\varphi_t}$ and suitable singularity assumptions \cite[Theorem~1.3]{Chen16}.

We begin with an explicit scalar-parameter differentiation formula
for weights smooth up to the boundary. Its formulation in terms of
iterated projections will also be used for the boundary-distance
family.
Since no boundary regularity is assumed in the next two theorems,
$\varphi\in C^\infty(\overline\Omega\times D)$ means that
$\varphi$ extends smoothly to a neighborhood of
$\overline\Omega\times J$ for every compact interval $J\Subset D$.
In particular, its derivatives are uniformly bounded on
$\overline\Omega\times J$.
\begin{theorem}\label{th:formula-1}
Let $\Omega$ be a bounded domain in $\mathbb C^n$ and let
$D\subset\mathbb R$ be open. Suppose that $\varphi$ is real-valued
and belongs to $C^\infty(\overline\Omega\times D)$, and set
$\varphi_t=\varphi(\cdot,t)$. Then
$K_{\varphi_t}(z,w)$ depends smoothly on $t$. Moreover,
for every $(z,w,t)\in\Omega\times\Omega\times D$,
\begin{equation}\label{fo:derivation-1}
\partial_t^mK_{\varphi_t}(z,w)
=
\sum_{\alpha\in\Lambda_m}
\frac{m!}{\alpha!}
P_{\varphi_t,\alpha}
\bigl(K_{\varphi_t}(\cdot,w)\bigr)(z),
\ \ \  m=1,2,\ldots.
\end{equation}
\end{theorem}
Here $\Lambda_m$ is the set of ordered compositions
$\alpha=(\alpha_1,\ldots,\alpha_r)$ of $m$ into positive integers,
$\alpha!:=\prod_{j=1}^r\alpha_j!$, and the operators
$P_{\varphi_t,\alpha}$ are defined in Section~2 using complete
exponential Bell polynomials.
To state the resulting derivative estimates, we recall the
definition of Gevrey regularity.

\begin{definition}\label{de:Gevrey-s}
Let $D\subset\mathbb R^d$ be open and let $s\ge1$.
A function $f\in C^\infty(D)$ belongs to the Gevrey class $G^s(D)$ if,
for every compact set $K\Subset D$, there exist constants $C_K,R_K>0$
such that
\begin{equation}\label{eq:Gevrey-definition}
\sup_{x\in K}|\partial^\mu f(x)|
\le C_K R_K^{|\mu|}(|\mu|!)^s
\end{equation}
for every multi-index $\mu\in\mathbb N_0^d$.
In particular, $G^1(D)$ is the class of real-analytic functions.
\end{definition}

For open sets $V\subset\mathbb R^{d_1}$ and
$D\subset\mathbb R^{d_2}$, we write
$f\in G_y^s(V\times D)$ if $f\in C^\infty(V\times D)$ and, for every
compact set $K\Subset V\times D$, there exist $C_K,R_K>0$ such that
\begin{equation}\label{eq:partial-Gevrey}
\sup_{(x,y)\in K}|\partial_y^\mu f(x,y)|
\le C_K R_K^{|\mu|}(|\mu|!)^s
\end{equation}
for every $\mu\in\mathbb N_0^{d_2}$.
For a function smooth up to $\partial\Omega$, the notation
$G_t^s(\overline\Omega\times D)$ means that these parameter estimates
hold uniformly on $\overline\Omega\times J$ for every compact
$J\Subset D$.

Theorem~\ref{th:formula-1} yields the following preservation of Gevrey
regularity. This consequence also follows from
Pasternak-Winiarski's analytic dependence theorem by composition
with a Gevrey curve in $L^\infty(\Omega)$.
\begin{theorem}\label{th:Bounded-weight}
Let $\Omega$ be a bounded domain in $\mathbb C^n$ and let
$D\subset\mathbb R$ be open. Suppose that $\varphi$ is real-valued
and
\[
\varphi\in
C^\infty(\overline\Omega\times D)
\cap
G_t^s(\overline\Omega\times D)
\]
for some $s\ge1$. Then
\[
K_{\varphi_t}(z,w)
\in
G_t^s(\Omega\times\Omega\times D).
\]
\end{theorem}

Uniform control up to the boundary cannot simply be replaced by
interior parameter regularity, even when the weights are uniformly
bounded. On the unit disc $\mathbb D$, consider
\[
e^{-\varphi_t(z)}
=1+\frac{t^2}{(1-|z|^2)^2+t^2}.
\]
Then $-\log2\le\varphi_t\le0$,
$\varphi\in C^\infty(\mathbb D\times\mathbb R)
\cap G_t^1(\mathbb D\times\mathbb R)$, and, for each fixed $t$,
$\varphi_t$ extends smoothly to $\overline{\mathbb D}$.
Radial symmetry gives
\[
K_{\varphi_t}(0,0)
=\left(\int_{\mathbb D}e^{-\varphi_t}\,dA\right)^{-1}
=
\begin{cases}
\displaystyle\frac{1}{\pi\bigl(1+|t|\arctan(1/|t|)\bigr)},&t\ne0,\\[3mm]
\displaystyle\frac1\pi,&t=0.
\end{cases}
\]
Consequently,
$K_{\varphi_t}(0,0)=\pi^{-1}-\tfrac12|t|+O(t^2)$ as $t\to0$,
so the kernel is not differentiable at $0$.
This example leads to the following question.
\begin{question}\label{qu:boundary-smoothness}
Under what boundary control does prescribed Gevrey regularity of
$t\mapsto\varphi_t$ imply the same Gevrey regularity of
$K_{\varphi_t}$, locally uniformly in the spatial variables?
In particular, when does this implication hold for
boundary-degenerate densities?
\end{question}

We now turn to boundary-distance weights as an important test case for this
question. Set $\delta=\delta_\Omega$ and consider
$\varphi_t=-t\log\delta$, $t>-1$.
We write $A^2_t(\Omega)$, $K_t$, and $P_t$ for the corresponding
Bergman space, kernel, and projection, and put
\[
\|f\|_t^2:=\int_\Omega|f|^2\delta^t\,dV,\ \ \ 
\langle f,g\rangle_t:=\int_\Omega f\overline g\,\delta^t\,dV.
\]
We abbreviate $K_t(w,w)$ to $K_t(w)$ and
$K_{\varphi_t}(w,w)$ to $K_{\varphi_t}(w)$. For $t\ne0$, the weight $-t\log\delta$ is unbounded near
$\partial\Omega$.
For a bounded $C^2$ domain, the constant function belongs to
$A^2_t(\Omega)$ when $t>-1$. For $t\le-1$, one has
$A^2_t(\Omega)=\{0\}$ on every bounded domain
\cite[Theorem~1.9(1)]{ChenBergmanSpace}.

We use the H\"ormander-type weighted $L^2$ existence theorem in
\cite{ChenBergmanSpace} to obtain boundedness of the Bergman projection
on nearby weighted spaces. A holomorphic conjugation argument then
gives a local extension of the kernel sections with values in a
fixed weighted Bergman space. The resulting norm control yields
the analytic estimates and justifies the derivative formula in the
following theorem.

\begin{theorem}\label{th:Gevrey-1}
Let $\Omega$ be a bounded pseudoconvex domain in $\mathbb C^n$ with
$C^2$ boundary.
For every compact interval $J\Subset(-1,\infty)$, there exists
$R_J>0$ such that, for every $S\Subset\Omega$ and
$\ell\in\mathbb N_0$,
\begin{equation}\label{eq:analytic-main}
\sup_{t\in J}\|\partial_t^mK_t\|_{C^\ell(S\times S)}
\le C_{J,S,\ell}R_J^{-m}m!,\ \ \ m\in\mathbb N_0.
\end{equation}
In particular,
$K_t\in G_t^1\bigl(\Omega\times\Omega\times(-1,\infty)\bigr)$:
the dependence on $t$ is real analytic, locally uniformly together
with all spatial derivatives. The singular-weight derivative formula
\eqref{fo:derivation} also holds.
\end{theorem}
Boundary-distance weights also arise in boundary asymptotics.
Building on work of Fefferman and Boutet de Monvel--Sj\"ostrand
\cite{Fefferman74,Fefferman79,BS75}, Hirachi studied the relation
between weighted Bergman and Szeg\H{o} kernel expansions for powers
of a smooth defining function on strictly pseudoconvex domains
\cite{Hirachi2004}. Engli\v{s} obtained power--logarithmic expansions
and meromorphic continuation in the exponent in this smooth-boundary
setting \cite{Englis2010,EnglisAnalytic2010}; Chen related the
endpoint behavior of distance-weighted kernels to the Szeg\H{o}
kernel \cite{ChenBergmanSpace}. By contrast,
Theorem~\ref{th:Gevrey-1} treats the Euclidean boundary-distance
weight on bounded pseudoconvex domains with $C^2$ boundary and
establishes local analyticity for $t>-1$ by weighted $L^2$ methods.

Bergman kernels also arise in large tensor-power asymptotics.
The theory initiated by Tian was developed by Catlin, Zelditch,
Lu, Ma--Marinescu, and Berman--Berndtsson--Sj\"ostrand
\cite{Tian90,Catlin99,Zelditch98,Lu00,MaMarinescu2007,MaMarinescu2008,BBS08}.
Zelditch's proof uses the Boutet de Monvel--Sj\"ostrand parametrix.
These results concern a large-parameter asymptotic regime, whereas
our results concern local regularity at finite parameter values.

The paper is organized as follows. Section~2 introduces the
Bell-polynomial notation, recalls elementary facts about holomorphic
operator families, and establishes the weighted openness properties
used later. Section~3 proves
Theorems~\ref{th:formula-1} and \ref{th:Bounded-weight}.
Section~\ref{sec:analyticity} first proves holomorphy in a fixed
weighted norm and then derives the local uniform analytic estimates.
Section~\ref{sec:logarithmic} identifies all parameter derivatives
by the Leibniz and composition argument of Section~3 and completes
the proof of Theorem~\ref{th:Gevrey-1}.

\section{Preliminaries}
We write $\mathbb N_0=\{0,1,2,\ldots\}$ and
$\mathbb N_+=\{1,2,\ldots\}$. 
Real parameters are denoted by $t,t_0,t_1,\ldots$, and complex
parameters by $\tau,\tau_0,\tau_1,\ldots$. 
\subsection{Bell polynomials}

Let $f$ and $g$ be $C^\infty$ functions such that
the composition $h=g\circ f$ is well defined. Fa\`a di Bruno's formula takes the following form (cf. \cite{Fraenkel78}):
\begin{eqnarray*}
h^{(m)}(x)=\sum_{k=1}^m g^{(k)}(f(x))\cdot B_{m,k}(f'(x),f''(x),\ldots,f^{(m-k+1)}(x)),
\ \ \  m=1,2,\ldots,
\end{eqnarray*}
where
\[
B_{m,k}(x_1,x_2,\ldots,x_{m-k+1})=\sum \frac{m!}{n_1!n_2!\cdots n_{m-k+1}!}\prod_{j=1}^{m-k+1}\left(\frac{x_j}{j!}\right)^{n_j},
\]
and the summation is taken over all nonnegative integers $(n_1,\ldots,n_{m-k+1})$ satisfying
\[
 n_1+n_2+\cdots+n_{m-k+1}=k
\]
and
\[
1\cdot n_1+2\cdot n_2+\cdots+ (m-k+1)\cdot n_{m-k+1}=m.
\]
In particular, when $h=e^f$, we have
\[
h^{(m)}(x)=e^{f(x)}\sum_{k=1}^m B_{m,k}(f'(x),f''(x),\ldots,f^{(m-k+1)}(x))=:e^{f(x)}\,B_m(f',\ldots,f^{(m)}),
\]
where $B_m$ is the $m$th complete exponential Bell polynomial.

If there exist constants $C,R>0$ and $s\ge1$ such that
$|x_j|\le CR^j (j!)^s$ for $j=1,2,\ldots,m$, then
for every term occurring in $B_{m,k}$,
\begin{eqnarray*}
\prod_{j=1}^{m-k+1}
\left|\frac{x_j}{j!}\right|^{n_j}
&\le&
C^kR^m
\prod_{j=1}^{m-k+1}(j!)^{(s-1)n_j}\\
&\le&
C^kR^m(m!)^{s-1},
\end{eqnarray*}
since $(p+q)!\ge p!q!$ for nonnegative integers $p,q$. Set $C_1:=\max\{1,C\}$. It follows that
\begin{eqnarray}\label{eq:Bell-polynomial}
|B_m(x_1,\ldots,x_m)|
&\le&
C_1^mR^m(m!)^s
\sum_{k=1}^m
\sum
\frac{1}{n_1!n_2!\cdots n_{m-k+1}!}\nonumber\\
&\le&
e(2C_1R)^m(m!)^s.
\end{eqnarray}
Indeed,
\[
1+\sum_{m\ge1}
\left(
\sum_{k=1}^m
\sum
\frac{1}{n_1!n_2!\cdots n_{m-k+1}!}
\right)t^m
=
\exp\left(\frac{t}{1-t}\right),
\]
and evaluation at $t=1/2$ gives the stated bound.
For an ordered tuple $\alpha=(\alpha_1,\ldots,\alpha_r)\in\mathbb N_+^r$ with $r\in\mathbb N_+$, we introduce
\[
|\alpha|:=\alpha_1+\cdots+\alpha_r,\ \ \  \alpha!:=\prod_{j=1}^r \alpha_j!,
\]
and
\[
\Lambda_{m,r}:=\{\alpha\in\mathbb N_+^r:|\alpha|=m\},
\ \ \ 
\Lambda_m:=\bigcup_{r=1}^m\Lambda_{m,r},
\ \ \  m\in\mathbb N_+.
\]
Note that $\#\Lambda_m=2^{m-1}$ for $m\ge1$. We also set
$\Lambda_0=\{\varnothing\}$, $|\varnothing|=0$, and
$\varnothing!=1$.

Motivated by Fa\`a di Bruno's formula, we define
\begin{eqnarray*}
&&P_{\varphi_t,k}(f)(z)
:=
P_{\varphi_t}
\left(-f B_k(-\varphi_t)\right)(z),\ \ \  k\in\mathbb N_+,\\
&&P_{\varphi_t,\alpha}:=P_{\varphi_t,\alpha_r}\circ\cdots\circ P_{\varphi_t,\alpha_1},
\ \ \  P_{\varphi_t,\varnothing}:=\mathrm{Id},
\end{eqnarray*}
provided that the right-hand side is well defined, where
\[
B_k(-\varphi_t)
:=B_k\left(\partial_t(-\varphi_t),\ldots,
                         \partial_t^k(-\varphi_t)\right),
\ \ \ 
(-\varphi_t)^{(j)}=\partial_t^j(-\varphi_t).
\]

\subsection{Holomorphic operator families}

For a Hilbert space $H$, $\mathcal L(H)$ denotes its bounded linear
operators and $\mathrm{Id}$ denotes the identity operator, with
\[
\|A\|_{\mathcal L(H)}
 =\sup_{\|u\|_H\le1}\|Au\|_H,\ \ \ 
\langle Au,v\rangle_H=\langle u,A^*v\rangle_H.
\]
Let $U\subset\mathbb C$ be open.
A map $F:U\to H$ is called holomorphic if, for every $\tau\in U$,
the limit
\[
F'(\tau):=
\lim_{\substack{\Delta\tau\to0\\\Delta\tau\in\mathbb C\setminus\{0\}}}
\frac{F(\tau+\Delta\tau)-F(\tau)}{\Delta\tau}
\]
exists in the norm of $H$.
For $F:U\to\mathcal L(H)$, holomorphy is defined by the same
condition, with convergence in operator norm.

The following elementary projection identity is related to the
Kerzman--Stein formula \cite{KerzmanStein78}.
We include the proof of the precise form used below.

\begin{lemma}\label{le:ra-projection-identity}
Let $A\in\mathcal L(H)$ satisfy $A^2=A$, and let $Q$ be the
orthogonal projection onto $\operatorname{Ran}A$. Then
$A+A^*-\mathrm{Id}$ is invertible,
\[
\|(A+A^*-\mathrm{Id})^{-1}\|\le1,
\]
and
\begin{equation}\label{ra:abstract-projection}
Q=A(A+A^*-\mathrm{Id})^{-1}.
\end{equation}
\end{lemma}

\begin{proof}
The range is closed because
$\operatorname{Ran}A=\ker(\mathrm{Id}-A)$.
Set $X=A+A^*-\mathrm{Id}$. Direct multiplication gives
\begin{align*}
X^*&=X,\\
X^2&=\mathrm{Id}+(A-A^*)^*(A-A^*),\\
\|Xu\|^2&=\|u\|^2+\|(A-A^*)u\|^2\ge\|u\|^2.
\end{align*}
Hence $X$ is injective and has closed range. Since
\[
(\operatorname{Ran}X)^\perp=\ker X^*=\ker X=\{0\},
\]
its range is also dense, and therefore equals $H$.
The same inequality gives $\|X^{-1}\|\le1$.
Finally,
\[
QA=A,\ \ \  AQ=Q,\ \ \  QA^*=Q,
\]
so
\[
Q(A+A^*-\mathrm{Id})=A+Q-Q=A.
\]
Multiplication by $X^{-1}$ proves \eqref{ra:abstract-projection}.
\end{proof}

\begin{lemma}[{cf. \cite[Section~2.3]{EnglisAnalytic2010}}]
\label{le:ra-weak-holomorphy}
Let $H$ be a complex Hilbert space, $U\subset\mathbb C$ open,
and $F:U\to\mathcal L(H)$ locally bounded.
Suppose that $\tau\mapsto\langle F(\tau)u,v\rangle_H$ is holomorphic
for all $u,v$ in a fixed dense linear subspace of $H$.
Then $F$ is holomorphic in operator norm.
Moreover, if $R>0$ and $\{|\tau-\tau_0|\le R\}\Subset U$, then
\begin{equation}\label{ra:operator-coefficients}
\|F(\tau)-F(\tau_0)\|
\le
\frac{|\tau-\tau_0|}{R-|\tau-\tau_0|}
\sup_{|\tau_1-\tau_0|=R}\|F(\tau_1)\|,
\ \ \  |\tau-\tau_0|<R.
\end{equation}
\end{lemma}

\begin{proof}
For $u,v\in H$, choose $u_j,v_j$ in the stated dense subspace
such that $u_j\to u$ and $v_j\to v$. For every $E\Subset U$,
\begin{align*}
&\sup_{\tau\in E}
\left|
\langle F(\tau)u_j,v_j\rangle_H
-\langle F(\tau)u,v\rangle_H
\right|\\
&\quad\le
\sup_{\tau\in E}\|F(\tau)\|
\bigl(
\|u_j-u\|\|v_j\|
+\|u\|\|v_j-v\|
\bigr)
\longrightarrow0.
\end{align*}
Thus $\tau\mapsto\langle F(\tau)u,v\rangle_H$ is holomorphic
for all $u,v\in H$. By
\cite[Section~2.3]{EnglisAnalytic2010},
$F$ is holomorphic in operator norm.

The Cauchy formula gives
\begin{equation}\label{ra:cauchy-operator}
F(\tau)-F(\tau_0)
=
\frac{\tau-\tau_0}{2\pi i}
\int_{|\tau_1-\tau_0|=R}
\frac{F(\tau_1)}{(\tau_1-\tau)(\tau_1-\tau_0)}\,d\tau_1,
\ \ \  |\tau-\tau_0|<R.
\end{equation}
Since $|\tau_1-\tau|\ge R-|\tau-\tau_0|$ on the circle,
\[
\|F(\tau)-F(\tau_0)\|
\le
\frac{|\tau-\tau_0|}{R-|\tau-\tau_0|}
\sup_{|\tau_1-\tau_0|=R}\|F(\tau_1)\|,
\]
which proves \eqref{ra:operator-coefficients}.
\end{proof}

\begin{lemma}[{cf. \cite[Section~2.3]{EnglisAnalytic2010}}]
\label{le:ra-inversion}
Let $F(\tau)$ be holomorphic with values in $\mathcal L(H)$ near
$\tau_0$, and suppose that $F(\tau_0)$ is invertible. Whenever
\[
\|(F(\tau)-F(\tau_0))F(\tau_0)^{-1}\|\le \theta<1,
\]
one has
\begin{align}
F(\tau)^{-1}
 &=F(\tau_0)^{-1}
   \sum_{j=0}^\infty
       \big[-(F(\tau)-F(\tau_0))F(\tau_0)^{-1}\big]^j,
\label{ra:neumann-series}\\
\|F(\tau)^{-1}\|&\le\frac{\|F(\tau_0)^{-1}\|}{1-\theta}.
\notag
\end{align}
The inverse is holomorphic in operator norm.
\end{lemma}

\subsection{Openness properties}
For convenience, we define $L^2_t(\Omega,\varphi)$ to be the Hilbert
space of measurable functions $f$ such that
\[
 \int_\Omega |f|^2e^{-\varphi}\delta^t<\infty.
\]
We write $L^2_t(\Omega)=L^2_t(\Omega,0)$.
We recall the following H\"ormander-type $L^2$ estimate for $\bar\partial$.

\begin{theorem}[cf. \cite{ChenBergmanSpace}]\label{th:Chen}
Let $\Omega$ be a bounded pseudoconvex domain with $C^2$ boundary and let $\varphi$ be a $C^2$ strictly plurisubharmonic (psh) function on $\Omega$. Then, for any $t>-1$ and any $\bar\partial$-closed $(0,1)$-form $v$ satisfying
$
\int_\Omega |v|^2_{i\partial\bar\partial\varphi}e^{-\varphi}\delta^t<\infty,
$
there exists a solution $u$ of $\bar\partial u=v$ such that
\[
 \int_\Omega |u|^2e^{-\varphi}\delta^t\le C(t) \int_\Omega |v|^2_{i\partial\bar\partial\varphi}e^{-\varphi}\delta^t.
\]
The constants may be chosen so that $\sup_{t\in J}C(t)<\infty$
for every compact interval $J\Subset(-1,\infty)$.
\end{theorem}

\begin{remark}[Uniformity of Chen's estimate]
Fix a compact interval $J\Subset(-1,\infty)$ and choose
$0<a<\beta<1$ and $b\ge0$ with $J\subset[-a,b]$.
For $t=-\alpha$, $0<\alpha\le a$, use one finite local
Diederich--Forn\ae ss cover, including an interior patch, as in
\cite[Section~2]{ChenBergmanSpace}, with negative $C^2$
plurisubharmonic functions $\rho_j$ satisfying
$-\rho_j\asymp\delta^\beta$ on the corresponding patches.
Set
\[
\eta_{j,\alpha}=(-\rho_j)^{\alpha/\beta},\ \ \ 
c(s)=\frac{\beta-a}{a}s^{-1}.
\]
Then $\eta_{j,\alpha}\asymp\delta^\alpha$ uniformly in $\alpha$,
$\eta_{j,\alpha}+c(\eta_{j,\alpha})^{-1}
=\beta\eta_{j,\alpha}/(\beta-a)$, and
\begin{align*}
-i\bigl(\partial\bar\partial\eta_{j,\alpha}
 +c(\eta_{j,\alpha})\partial\eta_{j,\alpha}
             \wedge\bar\partial\eta_{j,\alpha}\bigr)
=\frac{\alpha\eta_{j,\alpha}}{\beta}
 \left[
 \frac{i\partial\bar\partial\rho_j}{-\rho_j}
 +\left(1-\frac{\alpha}{a}\right)
   \frac{i\partial\rho_j\wedge\bar\partial\rho_j}{(-\rho_j)^2}
 \right]\ge0.
\end{align*}
Because the cover and a subordinate partition of unity are fixed,
while the comparison constants above and the derivatives of the
cutoffs are uniform for $0<\alpha\le a$, the constants produced by
the localization and product-rule estimates in
\cite[(2.2)--(2.3)]{ChenBergmanSpace} are independent of
$\alpha,\varepsilon,\varphi$, and $T$.
Put $\varphi_T=\varphi+T|z|^2$. On
$\Omega_\varepsilon=\{\delta>\varepsilon\}$, for sufficiently small
$\varepsilon>0$, Chen's localization
argument \cite[(2.1)--(2.3)]{ChenBergmanSpace} therefore gives
\begin{equation}\label{eq:Chen-uniform-basic}
\begin{aligned}
\int_{\Omega_\varepsilon}
 \langle i\partial\bar\partial\varphi_Tu,u\rangle
 e^{-\varphi_T}\delta^\alpha\,dV
\le C_J\int_{\Omega_\varepsilon}
 \bigl(|\bar\partial u|^2
       +|\bar\partial_{\varphi_T}^{*}u|^2+|u|^2\bigr)
 e^{-\varphi_T}\delta^\alpha\,dV
\end{aligned}
\end{equation}
for the $C^1$ test $(0,1)$-forms satisfying the
$\bar\partial$-Neumann boundary condition on
$\partial\Omega_\varepsilon$. Here $C_J$ depends only on $J$ and
$\Omega$, and is independent of $\alpha,\varepsilon,\varphi,T$.
Choosing $T>C_J$ absorbs the last term, leaving
$i\partial\bar\partial\varphi$ on the left. Chen's duality and
weak-limit argument then gives a uniform bound for $C(-\alpha)$;
comparison of $e^{-\varphi_T}$ and $e^{-\varphi}$ introduces only
the fixed factor $\exp(T\sup_\Omega|z|^2)$.

For $0\le t\le b$, fix a negative Diederich--Forn\ae ss function
$\rho$ with $c_1\delta^\kappa\le-\rho\le c_2\delta^\kappa$.
Applying H\"ormander's estimate to
$\Phi_t=\varphi-(t/\kappa)\log(-\rho)$ gives
$C(t)\le(c_2/c_1)^{t/\kappa}$ by comparison of the densities.
Hence $\sup_{t\in J}C(t)<\infty$.
\end{remark}
As a consequence of Theorem \ref{th:Chen}, we obtain the following Donnelly--Fefferman--Berndtsson-type estimate.

\begin{proposition}\label{prop:DFB}
Let $\Omega$ be a bounded pseudoconvex domain with $C^2$ boundary and let $\varphi$ be a $C^2$ psh function on $\Omega$. Let $\psi$ be a bounded $C^2$ strictly psh function on $\Omega$ satisfying
\begin{equation}\label{eq:DF}
|\bar\partial \psi|^2_{i\partial\bar\partial\psi}\le 1.
\end{equation}
Then, for any $t>-1$ and any $\bar\partial$-closed $(0,1)$-form $v$ for which the right-hand side below is finite, the $L^2_t(\Omega,\varphi)$-minimal solution of $\bar\partial u=v$ satisfies
\begin{equation}\label{eq:DonnellyFefferman}
 \int_\Omega |u|^2e^{-\varphi+\lambda\psi}\delta^t\le 4C(t)\int_\Omega |v|^2_{i\partial\bar\partial(\varphi+\lambda\psi)}e^{-\varphi+\lambda\psi}\delta^t,\ \ \  0<\lambda\le \frac1{4C(t)}.
\end{equation}
\end{proposition}

\begin{proof}
Fix $0<\lambda\le \frac{1}{4C(t)}$. We use the argument of Berndtsson--Charpentier \cite{BerndtssonCharpentier00}. Since $u\perp \ker\bar\partial$ in $L^2_t(\Omega,\varphi)$ and $\psi$ is bounded, we have $ue^{\lambda\psi}\perp \ker\bar\partial$ in $L^2_t(\Omega,\varphi+\lambda\psi)$. It follows from Theorem \ref{th:Chen} and \eqref{eq:DF} that
\begin{eqnarray*}
\int_\Omega |u|^2 e^{-\varphi+\lambda\psi}\delta^t & \le & C(t)\int_\Omega |\bar\partial(ue^{\lambda\psi})|^2_{i\partial\bar\partial(\varphi+\lambda\psi)} e^{-\varphi-\lambda\psi}\delta^t\\
& \le & 2 C(t)\left(\int_\Omega |v|^2_{i\partial\bar\partial(\varphi+\lambda\psi)}e^{-\varphi+\lambda\psi}\,\delta^t+\lambda \int_\Omega |u|^2 e^{-\varphi+\lambda\psi}\delta^t\right),
\end{eqnarray*}
from which \eqref{eq:DonnellyFefferman} follows immediately.
\end{proof}

We next establish the following openness property for $P_t$.

\begin{theorem}\label{th:Projection}
Let $\Omega$ be a bounded pseudoconvex domain with $C^2$ boundary. Given $t_0>-1$, there exist $\eta_0>0$ and $C_0>0$ such that, for any $0\leq\eta\leq \eta_0$, $t\in [t_0-\eta_0,t_0+\eta_0]$, and $f\in L^2_t(\Omega)$,
\[
 \int_\Omega |P_t(f)|^2\delta^{t-\eta}\le C_0 \int_\Omega |f|^2 \delta^{t-\eta}.
\]
\end{theorem}

\begin{proof}
By a theorem of Diederich--Forn\ae ss \cite{DiederichFornaess77}, there exists a negative $C^2$ strictly psh function $\rho$ on $\Omega$ such that
$
c_1\delta^\kappa\le -\rho \le c_2\delta^\kappa
$
for suitable $\kappa,c_1,c_2>0$; enlarging $c_2$ if necessary, assume
$c_2\geq c_1$.

Choose $0<\eta_*<(t_0+1)/2$ and set
\[
J_*=[t_0-\eta_*,t_0+\eta_*],\ \ \ 
C_*:=\sup_{t\in J_*}C(t)<\infty,\ \ \ 
\lambda_0:=\min\left\{1,\frac{1}{4C_*}\right\}.
\]
The finiteness of $C_*$ follows from Theorem~\ref{th:Chen} and its
uniformity remarks, and this choice makes
$0<\lambda\leq1/(4C(t))$ whenever $t\in J_*$ and
$0<\lambda\leq\lambda_0$.
For fixed $\varepsilon>0$, set
\[
\psi=\psi_\varepsilon=-\log(-\rho+\varepsilon),
\]
which satisfies \eqref{eq:DF}. For $\lambda>0$, let $P_{t,\lambda\psi}$ denote the Bergman projection on $L^2_t(\Omega,\lambda\psi)$.
We first take $f\in C_c^\infty(\Omega)$. For fixed $\varepsilon>0$,
the function $\psi_\varepsilon$ is bounded, so every expression below
belongs to the indicated $L^2$ space. The final estimate will be
extended to $L^2_{t-\eta}(\Omega)$ by density.

Since $\psi$ is bounded on $\Omega$, for any $g\in A^2_t(\Omega)$ we have
\begin{eqnarray*}
\int_\Omega P_t(f)\cdot\bar{g}\delta^t & = & \int_\Omega f\cdot \bar{g} \delta^t = \int_\Omega (e^{\lambda\psi} f)\cdot \bar{g} e^{-\lambda\psi} \delta^t\\
& = & \int_\Omega P_{t,\lambda\psi}(e^{\lambda\psi} f)\cdot\bar{g}e^{-\lambda\psi}\delta^t\\
& = & \int_\Omega P_t(e^{-\lambda\psi}P_{t,\lambda\psi}(e^{\lambda\psi} f))\cdot\bar{g}\delta^t.
\end{eqnarray*}
By uniqueness of the orthogonal projection,
\[
P_t(f)=P_t(e^{-\lambda\psi}P_{t,\lambda\psi}(e^{\lambda\psi}f)),\ \ \  f\in L^2_t(\Omega).
\]
Set $h=e^{-\lambda\psi}P_{t,\lambda\psi}(e^{\lambda\psi}f)$. Since $P_t(h)=h-u_t$, where $u_t$ is the $L^2_t(\Omega)$-minimal solution of $\bar\partial u=\bar\partial h$, Proposition \ref{prop:DFB} gives
\begin{eqnarray*}
\int_\Omega |P_t(f)|^2 e^{\lambda\psi}\delta^t & = & \int_\Omega |P_t(h)|^2 e^{\lambda\psi}\delta^t\\
&\le & 2 \int_\Omega |h|^2 e^{\lambda\psi}\delta^t+ 2\int_\Omega |u_t|^2 e^{\lambda\psi}\delta^t\\
& \le & 2\int_\Omega |P_{t,\lambda\psi}(e^{\lambda\psi}f)|^2 e^{-\lambda\psi}\delta^t
 + C_0' \int_\Omega |\bar\partial h |^2_{\lambda i\partial\bar\partial\psi} e^{\lambda\psi}\delta^t
\end{eqnarray*}
This holds for every $t\in J_*$ and $0<\lambda\leq\lambda_0$,
with the uniform choice $C_0'=8C_*$: indeed, the second term above
is twice the left-hand side of \eqref{eq:DonnellyFefferman}.
Since $\bar\partial h=-\lambda h \bar\partial\psi$, we have $|\bar\partial h |^2_{\lambda i\partial\bar\partial\psi}\le \lambda |h|^2$, and hence
\[
\int_\Omega |P_t(f)|^2 e^{\lambda\psi}\delta^t \le (2+\lambda C_0')\int_\Omega |P_{t,\lambda\psi}(e^{\lambda\psi}f)|^2 e^{-\lambda\psi}\delta^t\le (2+\lambda C_0')\int_\Omega |f|^2 e^{\lambda\psi}\delta^t.
\]
For $f\in C_c^\infty(\Omega)$, the right-hand side stays finite as
$\varepsilon\downarrow0$. Since
$e^{\lambda\psi_\varepsilon}=(-\rho+\varepsilon)^{-\lambda}$
increases pointwise to $(-\rho)^{-\lambda}$, the monotone
convergence theorem on both sides gives
\[
\int_\Omega |P_t(f)|^2 (-\rho)^{-\lambda}\delta^t \le (2+\lambda C_0')\int_\Omega |f|^2 (-\rho)^{-\lambda}\delta^t.
\]
Choose $\eta_0>0$ so that
\[
\eta_0\le\min\{\kappa\lambda_0,\eta_*\},\ \ \  t_0-2\eta_0>-1,
\]
and enlarge
$C_0=(2+\lambda_0C_0')(c_2/c_1)^{\lambda_0}$ if necessary so that $C_0\ge1$.
For $0<\eta\le\eta_0$, take $\lambda=\eta/\kappa$.
By $c_1\delta^\kappa\le-\rho\le c_2\delta^\kappa$,
\[
\|P_tf\|_{t-\eta}^2
\le (2+\lambda C_0')\left(\frac{c_2}{c_1}\right)^\lambda
       \|f\|_{t-\eta}^2
\le C_0\|f\|_{t-\eta}^2,
\ \ \  f\in C_c^\infty(\Omega).
\]
Density extends this estimate to $L^2_{t-\eta}(\Omega)$.
The continuous inclusion $L^2_{t-\eta}(\Omega)\hookrightarrow L^2_t(\Omega)$
and the boundedness of $P_t$ on $L^2_t$ identify this extension
with the original projection. The case $\eta=0$ follows from the contractivity of $P_t$
on $L^2_t$.
\end{proof}

We also use the following openness property for the Bergman kernel.

\begin{theorem}\label{th:Openness}
Let $\Omega$ be a bounded pseudoconvex domain with $C^2$ boundary, and let $S$ be a compact subset of $\Omega$. Given $t_0>-1$, there exists $\eta_0>0$ such that, for any $w\in S$ and $t\in [t_0-\eta_0,t_0+\eta_0]$,
\[
\int_\Omega |K_t(\cdot,w)|^2 \delta^{t-\eta_0}\le \widetilde C_0,
\]
where $\widetilde C_0$ depends only on $S$, $\Omega$, and $t_0$.
\end{theorem}

\begin{proof}
For $\varepsilon>0$, set $\Omega_\varepsilon=\{z\in \Omega: \delta(z)>\varepsilon\}$. Fix $\varepsilon_0>0$ such that $S\subset \Omega_{\varepsilon_0}$, and let $\chi_0$ be the characteristic function of $\Omega_{\varepsilon_0}$. Let $K_{\Omega_{\varepsilon_0},t}$ denote the Bergman kernel for the density $\delta_\Omega^t\,dV$ on $\Omega_{\varepsilon_0}$. Note that
\[
\int_\Omega \chi_0 K_{\Omega_{\varepsilon_0},t}(\cdot,w) \overline{K_t(\cdot,z)}\delta^t
=
\overline{\int_{\Omega_{\varepsilon_0}} K_t(\cdot,z) \overline{K_{\Omega_{\varepsilon_0},t}(\cdot,w)}\delta^t}
=K_t(z,w).
\]
Applying Theorem \ref{th:Projection} with $f=\chi_0K_{\Omega_{\varepsilon_0},t}(\cdot,w)$, we obtain
\[
\int_\Omega |K_t(\cdot,w)|^2 \delta^{t-\eta_0}= \int_\Omega |P_t(f)|^2 \delta^{t-\eta_0}\le C_0 \int_\Omega |f|^2 \delta^{t-\eta_0}
=C_0 \int_{\Omega_{\varepsilon_0}} |K_{\Omega_{\varepsilon_0},t}(\cdot,w)|^2 \delta^{t-\eta_0}\le \widetilde C_0
\]
for all $t\in [t_0-\eta_0,t_0+\eta_0]$ and $w\in S$.
Indeed, balls $B(w,r)\subset\Omega_{\varepsilon_0}$ of a common radius
$r>0$ exist for all $w\in S$, so the submean inequality and the
uniform positive lower bound for $\delta^t$ on
$\Omega_{\varepsilon_0}$ give
$K_{\Omega_{\varepsilon_0},t}(w,w)\leq C$ uniformly in $w,t$;
the reproducing norm identity and
$\delta^{-\eta_0}\leq\varepsilon_0^{-\eta_0}$ then bound the last
integral by $\varepsilon_0^{-\eta_0}C$.
\end{proof}

\section{Weights smooth up to the boundary}
\subsection{Parameter dependence}
Fix $z,w\in\Omega$, $t_0\in D$, and
$[t_0-\eta_0,t_0+\eta_0]\Subset D$.
In the proof of Theorem~\ref{th:Bounded-weight}, choose $C_0,R>0$ so that,
uniformly for $t\in[t_0-\eta_0,t_0+\eta_0]$,
\begin{equation}\label{es:Gevrey}
\sup_{\zeta\in\overline\Omega}|\varphi_t^{(j)}(\zeta)|
\le C_0R^j(j!)^s,\ \ \  j\ge1.
\end{equation}
The smoothness and derivative formula in Theorem~\ref{th:formula-1}
require only the bounds for each fixed finite number of derivatives.
We first verify the lower-order regularity following
\cite{ChenLecture}.
For a parameter-dependent function $f_t$, we write
\[
f_t^{(j)}:=\partial_t^j f_t\quad(j\in\mathbb N_0),\ \ \ 
\Delta f_t:=f_t-f_{t_0},\ \ \  \Delta t:=t-t_0.
\]
In particular, $f_t^{(0)}=f_t$ and
$K_{\varphi_t}^{(j)}:=\partial_t^jK_{\varphi_t}$.

Step 1. {\it Lipschitz continuity}. By the reproducing property of the Bergman kernel,
\begin{eqnarray}\label{eq:key-1}
 K_{\varphi_t}(z,w) & = & \int_\Omega K_{\varphi_t}(\cdot,w) \overline{K_{\varphi_{t_0}}(\cdot,z)} e^{-\varphi_{t_0}}\nonumber\\
 &  = & \int_\Omega K_{\varphi_t}(\cdot,w) \overline{K_{\varphi_{t_0}}(\cdot,z)} e^{-\varphi_t} \nonumber\\
 && + \int_\Omega K_{\varphi_t}(\cdot,w) \overline{K_{\varphi_{t_0}}(\cdot,z)} (e^{-\varphi_{t_0}}-e^{-\varphi_t})\nonumber\\
 & =: & K_{\varphi_{t_0}}(z,w)+I.
\end{eqnarray}
Since the weighted norms $\|\cdot\|_{\varphi_t}$ and $\|\cdot\|_{\varphi_{t_0}}$ are uniformly equivalent on the fixed parameter interval,
\begin{eqnarray*}
|I|  &\le& C|\Delta t| \, \|K_{\varphi_t}(\cdot,w)\|_{\varphi_t}\, \|K_{\varphi_{t_0}}(\cdot,z)\|_{\varphi_{t_0}} \\
&=& C|\Delta t| \, K_{\varphi_t}(w)^{\frac12}\, K_{\varphi_{t_0}}(z)^{\frac12}\\
& \le& C|\Delta t|.
\end{eqnarray*}
Here and below, $C$ denotes a generic constant depending only on $z,w,\Omega,t_0$ and the fixed parameter interval.

Step 2. {\it $C^1$ differentiability}. We write
\begin{eqnarray*}
 I & = & \int_\Omega K_{\varphi_{t_0}}(\cdot,w) \overline{K_{\varphi_{t_0}}(\cdot,z)} (e^{-\varphi_{t_0}}-e^{-\varphi_t})\\
 & &+ \int_\Omega (K_{\varphi_t}(\cdot,w)- K_{\varphi_{t_0}}(\cdot,w)) \overline{K_{\varphi_{t_0}}(\cdot,z)}
 (e^{-\varphi_{t_0}}-e^{-\varphi_t})\\
 & =: & II+III.
\end{eqnarray*}
By \eqref{eq:key-1},
\begin{eqnarray*}
|III| & \le & C |\Delta t|\, \int_\Omega |K_{\varphi_t}(\cdot,w)- K_{\varphi_{t_0}}(\cdot,w)|\,|K_{\varphi_{t_0}}(\cdot,z)| e^{-\varphi_{t_0}}\\
& \le & C|\Delta t| \, \|K_{\varphi_t}(\cdot,w)-K_{\varphi_{t_0}}(\cdot,w)\|_{\varphi_{t_0}} \, \|K_{\varphi_{t_0}}(\cdot,z)\|_{\varphi_{t_0}},
\end{eqnarray*}
while
\begin{eqnarray}\label{eq:integ-1}
&& \int_\Omega |K_{\varphi_t}(\cdot,w)-K_{\varphi_{t_0}}(\cdot,w)|^2 e^{-\varphi_{t_0}}\nonumber\\
& = & \int_\Omega |K_{\varphi_t}(\cdot,w)|^2 e^{-\varphi_{t_0}} + \int_\Omega |K_{\varphi_{t_0}}(\cdot,w)|^2 e^{-\varphi_{t_0}}\nonumber\\
&& -2\operatorname{Re} \int_\Omega K_{\varphi_t}(\cdot,w)\overline{K_{\varphi_{t_0}}(\cdot,w)} e^{-\varphi_{t_0}}\nonumber\\
& = &K_{\varphi_t}(w)(1+O(|\Delta t|))+K_{\varphi_{t_0}}(w) - 2 K_{\varphi_t}(w).
\end{eqnarray}
It follows from Step 1 that
\[
 | III | \le C |\Delta t|^{3/2}.
\]
Therefore,
\begin{eqnarray*}
\lim_{t\to t_0} \frac{K_{\varphi_t}(z,w)-K_{\varphi_{t_0}}(z,w)}{\Delta t} & = & \lim_{t\to t_0} \int_\Omega K_{\varphi_{t_0}}(\cdot,w)\overline{K_{\varphi_{t_0}}(\cdot,z)} \, \frac{e^{-\varphi_{t_0}}-e^{-\varphi_t} }{\Delta t}\\
& = & \int_\Omega K_{\varphi_{t_0}}(\cdot,w) \overline{K_{\varphi_{t_0}}(\cdot,z)} \varphi_{t_0}'\, e^{-\varphi_{t_0}},
\end{eqnarray*}
by the dominated convergence theorem. Thus $K_{\varphi_t}(z,w)$ is $C^1$ in $t$, and
\begin{eqnarray}\label{eq:C^1-1}
K_{\varphi_t}^{(1)}(z,w)  &  = & \int_\Omega K_{\varphi_t}(\cdot,w) \overline{K_{\varphi_t}(\cdot,z)} \varphi_t' \,e^{-\varphi_t}\nonumber\\
& = & P_{\varphi_t}\left(-K_{\varphi_t}(\cdot,w)\,(-\varphi_t)'\right)(z)\nonumber\\
& = & P_{\varphi_t,1}\bigl(K_{\varphi_t}(\cdot,w)\bigr)(z).
\end{eqnarray}

Step 3. {\it $C^2$ differentiability}. By \eqref{eq:C^1-1},
\[
\Delta K_{\varphi_t}^{(1)}(z,w) = \int_\Omega K_{\varphi_t}(\cdot,w)\overline{K_{\varphi_t}(\cdot,z)}\varphi_t'\,e^{-\varphi_t}
- \int_\Omega K_{\varphi_{t_0}}(\cdot,w)\overline{K_{\varphi_{t_0}}(\cdot,z)}\varphi_{t_0}'\,e^{-\varphi_{t_0}}.
\]
Set $g_{t,1}:=\varphi_t'\,e^{-\varphi_t}$. Then
\begin{eqnarray*}
&& K_{\varphi_t}(\zeta,w)\overline{K_{\varphi_t}(\zeta,z)} \, g_{t,1}(\zeta) \\
& = & \left(K_{\varphi_{t_0}}(\zeta,w)+\Delta K_{\varphi_t}(\zeta,w)\right) \overline{(K_{\varphi_{t_0}}(\zeta,z)+\Delta K_{\varphi_t}(\zeta,z))}\,(g_{t_0,1}+\Delta g_{t,1})\\
& = & \left(K_{\varphi_{t_0}}(\zeta,w)\overline{K_{\varphi_{t_0}}(\zeta,z)}+ K_{\varphi_{t_0}}(\zeta,w)\overline{\Delta K_{\varphi_t}(\zeta,z)}+\Delta K_{\varphi_t}(\zeta,w)
\overline{K_{\varphi_{t_0}}(\zeta,z)}\right) g_{t_0,1}\\
&& + K_{\varphi_{t_0}}(\zeta,w) \overline{K_{\varphi_{t_0}}(\zeta,z)}\, \Delta g_{t,1} + \mathcal R(z,w,\zeta,t),
\end{eqnarray*}
where $\mathcal R(z,w,\zeta,t)$ denotes the remainder. Hence
\begin{eqnarray*}
\Delta K_{\varphi_t}^{(1)}(z,w) & = & \int_\Omega K_{\varphi_{t_0}}(\cdot,w)\overline{\Delta K_{\varphi_t}(\cdot,z)}\, g_{t_0,1} + \int_\Omega \Delta K_{\varphi_t}(\cdot,w) \overline{K_{\varphi_{t_0}}(\cdot,z)}\, g_{t_0,1}\\
& & + \int_\Omega K_{\varphi_{t_0}}(\cdot,w) \overline{K_{\varphi_{t_0}}(\cdot,z)}\, \Delta g_{t,1} + \int_\Omega \mathcal R(z,w,\cdot,t).
\end{eqnarray*}
By \eqref{eq:key-1},
\[
\Delta K_{\varphi_t}(\zeta,w) =-\int_\Omega K_{\varphi_t}(\cdot,w) \overline{K_{\varphi_{t_0}}(\cdot,\zeta)} (\Delta e^{-\varphi_t}).
\]
Therefore,
\begin{eqnarray*}
&& \frac{\Delta K_{\varphi_t}(\zeta,w)}{\Delta t} - K_{\varphi_{t_0}}^{(1)}(\zeta,w)\\
& = & \int_\Omega (K_{\varphi_{t_0}}(\cdot,w)-K_{\varphi_t}(\cdot,w)) \overline{K_{\varphi_{t_0}}(\cdot,\zeta)} \, \frac{\Delta e^{-\varphi_t}}{\Delta t}\\
&& + \int_\Omega K_{\varphi_{t_0}}(\cdot,w) \overline{K_{\varphi_{t_0}}(\cdot,\zeta)}\left(
\left.\partial_t e^{-\varphi_t}\right|_{t=t_0}
-\frac{\Delta e^{-\varphi_t}}{\Delta t}
\right).
\end{eqnarray*}
Since the norm of the Bergman projection is at most one,
\begin{align*}
&\left\|\frac{\Delta K_{\varphi_t}(\cdot,w)}{\Delta t}
       -K_{\varphi_{t_0}}^{(1)}(\cdot,w)\right\|_{\varphi_{t_0}}\\
&\quad\le
 \left\|\frac{\Delta e^{-\varphi_t}}{\Delta t}
                      e^{\varphi_{t_0}}\right\|_{L^\infty(\Omega)}
 \|K_{\varphi_t}(\cdot,w)-K_{\varphi_{t_0}}(\cdot,w)\|_{\varphi_{t_0}}\\
&\ \ \ +
 \left\|\left(
 \left.\partial_t e^{-\varphi_t}\right|_{t=t_0}
       -\frac{\Delta e^{-\varphi_t}}{\Delta t}
 \right)e^{\varphi_{t_0}}\right\|_{L^\infty(\Omega)}
 \|K_{\varphi_{t_0}}(\cdot,w)\|_{\varphi_{t_0}}
 \longrightarrow0 .
\end{align*}
Consequently,
\[
\int_\Omega \frac{\Delta K_{\varphi_t}(\cdot,w)}{\Delta t} \overline{K_{\varphi_{t_0}}(\cdot,z)}g_{t_0,1}  \to  \int_\Omega K_{\varphi_{t_0}}^{(1)}(\cdot,w) \overline{K_{\varphi_{t_0}}(\cdot,z)}g_{t_0,1}.
\]
Similarly,
\[
\int_\Omega K_{\varphi_{t_0}}(\cdot,w) \overline{\frac{\Delta K_{\varphi_t}(\cdot,z)}{\Delta t}} g_{t_0,1} \to  \int_\Omega K_{\varphi_{t_0}}(\cdot,w) \overline{K_{\varphi_{t_0}}^{(1)}(\cdot,z)}g_{t_0,1},
\]
\[
\int_\Omega K_{\varphi_{t_0}}(\cdot,w) \overline{K_{\varphi_{t_0}}(\cdot,z)} \, \frac{\Delta g_{t,1}}{\Delta t} \to  \int_\Omega K_{\varphi_{t_0}}(\cdot,w) \overline{K_{\varphi_{t_0}}(\cdot,z)} \left.\partial_t g_{t,1}\right|_{t=t_0},
\]
and $\int_\Omega \mathcal R(z,w,\cdot,t)/\Delta t\to 0$ as $t\to t_0$. Thus $K_{\varphi_t}(z,w)$ is $C^2$ in $t$, and
\begin{eqnarray*}
K_{\varphi_t}^{(2)}(z,w) & = & \int_\Omega K_{\varphi_t}(\cdot,w) \overline{K_{\varphi_t}^{(1)}(\cdot,z)}\varphi_t'\, e^{-\varphi_t} \\
& & +\int_\Omega K_{\varphi_t}^{(1)}(\cdot,w)  \overline{K_{\varphi_t}(\cdot,z)}\varphi_t' \,e^{-\varphi_t}\\
&& + \int_\Omega K_{\varphi_t}(\cdot,w) \overline{K_{\varphi_t}(\cdot,z)} (\varphi_t''-(\varphi_t')^2) \,e^{-\varphi_t} \\
& =: & I_1(z,w)+I_2(z,w)+I_3(z,w).
\end{eqnarray*}
Therefore,
\[
I_2(z,w)=P_{\varphi_t}\left( K_{\varphi_t}^{(1)}(\cdot,w) \varphi_t' \right) (z)=P_{\varphi_t,(1,1)}\bigl(K_{\varphi_t}(\cdot,w)\bigr)(z),
\]
\[
I_3(z,w)=P_{\varphi_t} \left(-K_{\varphi_t}(\cdot,w)B_2(-\varphi_t) \right)(z).
\]
Selfadjointness of $P_{\varphi_t}$ and the reality of $\varphi_t'$
give
\begin{align*}
I_1(z,w)
 &=\left\langle \varphi_t'K_{\varphi_t}(\cdot,w),
        P_{\varphi_t}(\varphi_t'K_{\varphi_t}(\cdot,z))
                      \right\rangle_{\varphi_t}\\
 &=\left\langle P_{\varphi_t}(\varphi_t'K_{\varphi_t}(\cdot,w)),
        \varphi_t'K_{\varphi_t}(\cdot,z)\right\rangle_{\varphi_t}\\
 &=P_{\varphi_t}
   \left(\varphi_t'P_{\varphi_t}
                 (\varphi_t'K_{\varphi_t}(\cdot,w))\right)(z)
 =P_{\varphi_t,(1,1)}(K_{\varphi_t}(\cdot,w))(z).
\end{align*}
All factors belong to $L^2(\Omega,\varphi_t)$ because $\varphi_t'$ is bounded.
Together with the formulas for $I_2$ and $I_3$, this agrees with \eqref{fo:derivation-1} for $m=2$.

Step 4. {\it Gevrey regularity}. Since $P_{\varphi_t}$ is an orthogonal projection and the multipliers defining $P_{\varphi_t,\alpha}$ are bounded for smooth weights, we have
\begin{eqnarray}\label{inq:smooth-estimate}
\|P_{\varphi_t,\alpha}(f)\|_{\varphi_t}\le C_{t_0,\alpha} \|f\|_{\varphi_t}<\infty,\ \ \  f\in L^2(\Omega,\varphi_t),
\end{eqnarray}
for $t\in[t_0-\eta_0,t_0+\eta_0]$. Here $C_{t_0,\alpha}$ may depend on $\alpha$; the uniform dependence on $|\alpha|$ required below follows from the Gevrey estimate \eqref{es:Gevrey}. We obtain the following lemma.
\begin{lemma}\label{le:Openness-1}
Assume \eqref{es:Gevrey}. For any $t\in[t_0-\eta_0,t_0+\eta_0]$, fixed $w\in\Omega$, and $\alpha\in\Lambda_{m,r}$,
\begin{eqnarray}\label{inq:gevrey-estimate}
\int_\Omega |P_{\varphi_t,\alpha}(K_{\varphi_t}(\cdot,w))|^2 \,e^{-\varphi_t}
\le C R_0^{2|\alpha|}(\alpha!)^{2s},
\end{eqnarray}
where $C$ and $R_0$ depend only on $w,t_0$ and the fixed parameter interval.
\end{lemma}
\begin{proof}
By \eqref{es:Gevrey},
\[
\left|\varphi_t^{(\alpha_j)}(z)\right|\le C_0 R^{\alpha_j} (\alpha_j!)^s,\ \ \  j=1,2,\ldots,r,
\]
for $t\in[t_0-\eta_0,t_0+\eta_0]$ and $z\in\Omega$. It follows from \eqref{eq:Bell-polynomial} that
\[
\left|B_{\alpha_j}(-\varphi_t)\right|\le R_0^{\alpha_j} (\alpha_j!)^s,\ \ \  j=1,2,\ldots,r,
\]
where $R_0=2e\max\{1,C_0\}R$.
Using the contractivity of the orthogonal projection,
\begin{eqnarray*}
\int_\Omega |P_{\varphi_t,\alpha}(K_{\varphi_t}(\cdot,w))|^2 \,e^{-\varphi_t}
&\le&
\int_\Omega |P_{\varphi_t,\alpha'}(K_{\varphi_t}(\cdot,w))|^2
|B_{\alpha_r}(-\varphi_t)|^2
 e^{-\varphi_t}\\
&\le&
R_0^{2\alpha_r} (\alpha_r!)^{2s}
 \int_\Omega |P_{\varphi_t,\alpha'}(K_{\varphi_t}(\cdot,w))|^2  \,e^{-\varphi_t},
\end{eqnarray*}
where $\alpha'=(\alpha_1,\ldots,\alpha_{r-1})$. Iterating this estimate gives
\begin{eqnarray*}
\int_\Omega |P_{\varphi_t,\alpha}(K_{\varphi_t}(\cdot,w))|^2 \,e^{-\varphi_t}
&\le&
K_{\varphi_t}(w)\prod_{j=1}^{r} R_0^{2\alpha_j} (\alpha_j!)^{2s}
=K_{\varphi_t}(w)R_0^{2|\alpha|}(\alpha!)^{2s}.
\end{eqnarray*}
By continuity of $K_{\varphi_t}(w)$ in $t$, it is uniformly bounded on $[t_0-\eta_0,t_0+\eta_0]$.
\end{proof}

We also use the following commutation identity.
\begin{lemma}\label{le:commutation-1}
Let $\chi_V$ be the characteristic function of a measurable set $V\subset\Omega$, and let $f\in A^2(\Omega,\varphi_t)$. For any ordered composition $\alpha\in\Lambda_{m,r}$ and $k\in\mathbb N_+$,
\begin{eqnarray*}
\int_\Omega
\chi_V f\cdot \overline{P_{\varphi_t,\alpha}(K_{\varphi_t}(\cdot,z))} B_{k}(-\varphi_t)
\,e^{-\varphi_t}
=-P_{\varphi_t,\alpha^-}\circ P_{\varphi_t,k}(\chi_V f)(z),
\end{eqnarray*}
where $\alpha^-=(\alpha_r,\alpha_{r-1},\ldots,\alpha_1)$ and $t\in [t_0-\eta_0,t_0+\eta_0]$.
\end{lemma}
\begin{proof}
For $g\in L^2(\Omega,\varphi_t)$ and
$h\in A^2(\Omega,\varphi_t)$, the multiplier
$B_j(-\varphi_t)$ is bounded, so
\[
|\langle gB_j(-\varphi_t),h\rangle_{\varphi_t}|
\le \|B_j(-\varphi_t)\|_{L^\infty(\Omega)}
   \|g\|_{\varphi_t}\|h\|_{\varphi_t}<\infty.
\]
Since $P_{\varphi_t}h=h$ and
$P_{\varphi_t,j}g=-P_{\varphi_t}(gB_j(-\varphi_t))$,
\[
\langle gB_j(-\varphi_t),h\rangle_{\varphi_t}
=\langle P_{\varphi_t}(gB_j(-\varphi_t)),h\rangle_{\varphi_t}
=-\langle P_{\varphi_t,j}g,h\rangle_{\varphi_t}.
\]
For $u,v\in A^2(\Omega,\varphi_t)$, the reality of
$B_j(-\varphi_t)$ also gives
\[
\langle P_{\varphi_t,j}u,v\rangle_{\varphi_t}
=-\langle B_j(-\varphi_t)u,v\rangle_{\varphi_t}
=-\langle u,B_j(-\varphi_t)v\rangle_{\varphi_t}
=\langle u,P_{\varphi_t,j}v\rangle_{\varphi_t}.
\]
In the first transfer below, take $g=\chi_Vf$ and
$h=P_{\varphi_t,\alpha}(K_{\varphi_t}(\cdot,z))$;
the subsequent transfers use the second identity.
Consequently,
\begin{align*}
&\left\langle \chi_V f B_k(-\varphi_t),
 P_{\varphi_t,\alpha}(K_{\varphi_t}(\cdot,z))
 \right\rangle_{\varphi_t}\\
&\quad=-\left\langle P_{\varphi_t,k}(\chi_V f),
 P_{\varphi_t,\alpha}(K_{\varphi_t}(\cdot,z))
 \right\rangle_{\varphi_t}\\
&\quad=-\left\langle
 P_{\varphi_t,\alpha^-}\circ P_{\varphi_t,k}(\chi_V f),
 K_{\varphi_t}(\cdot,z)\right\rangle_{\varphi_t}\\
&\quad=-P_{\varphi_t,\alpha^-}\circ P_{\varphi_t,k}(\chi_V f)(z).
\end{align*}
Each transfer is justified by \eqref{inq:smooth-estimate}.
The identity also holds for the empty $\alpha$, with its composition
equal to the identity.
\end{proof}

Assuming \eqref{fo:derivation-1}, we first derive the Gevrey estimates
in Theorem~\ref{th:Bounded-weight}. For $m\ge1$,
Lemma~\ref{le:Openness-1} gives
\begin{eqnarray*}
\left\| K_{\varphi_t}^{(m)}(\cdot,w)\right\|_{\varphi_t}
&\le&
\sum_{\alpha\in\Lambda_m}\frac{m!}{\alpha!} \|P_{\varphi_t,\alpha}(K_{\varphi_t}(\cdot,w))\|_{\varphi_t}\\
&\le&
\sum_{\alpha\in\Lambda_m}m! \sqrt{C}R_0^{|\alpha|}(\alpha!)^{s-1}
\le \sqrt{C}(2R_0)^m (m!)^s.
\end{eqnarray*}
Indeed, $|\alpha|=m$, $\alpha!\le m!$, and $\#\Lambda_m=2^{m-1}$. Moreover, by boundedness of point evaluation,
\[
\left|
K_{\varphi_t}^{(m)}(z,w)
\right|
\le
K_{\varphi_t}(z,z)^{1/2}
\left\|
K_{\varphi_t}^{(m)}(\cdot,w)
\right\|_{\varphi_t}
\le C_{z,w,t_0}R_1^m(m!)^s.
\]
For every $S\Subset\Omega$, the uniform equivalence of the weighted
norms gives
\[
\sup_{\substack{t\in[t_0-\eta_0,t_0+\eta_0]\\w\in S}}
K_{\varphi_t}(w)<\infty.
\]
Hence the preceding estimates hold uniformly for $z,w\in S$.

Interior Cauchy estimates on a slightly larger compact set give the
corresponding bounds for all spatial derivatives of
$K_{\varphi_t}^{(m)}$. Once the induction below establishes the
existence and continuity in $t$ of these parameter derivatives in a
fixed weighted $L^2$ norm, interior evaluation and Cauchy estimates
give joint continuity of every mixed space--parameter derivative on
compact subsets of $\Omega\times\Omega\times D$. Hence the kernel is
jointly smooth, and it remains to prove \eqref{fo:derivation-1} for
every positive integer $m$.

\subsection{Proof of Theorem \ref{th:formula-1}}
We prove \eqref{fo:derivation-1} together with the following
Leibniz formula and convergence statement by induction:
\begin{eqnarray}\label{fo:Leibniz-1}
K_{\varphi_t}^{(m)}(z,w)
=
-\sum\frac{(m-1)!}{m_1!m_2!m_3!}
\int_\Omega
B_{m_3+1}(-\varphi_t)
K_{\varphi_t}^{(m_1)}(\cdot,w)
\overline{K_{\varphi_t}^{(m_2)}(\cdot,z)}
\,e^{-\varphi_t},
\end{eqnarray}
and
\begin{eqnarray}\label{fo:estimate-1}
\left\|
\frac{\Delta K_{\varphi_t}^{(m-2)}(\cdot,w)}{\Delta t}
-K_{\varphi_{t_0}}^{(m-1)}(\cdot,w)
\right\|_{\varphi_{t_0}}
\longrightarrow0
\ \ \ (t\to t_0),\quad m\ge3,
\end{eqnarray}
where the sum in \eqref{fo:Leibniz-1} is taken over
$(m_1,m_2,m_3)\in\mathbb N_0^3$ satisfying
$m_1+m_2+m_3=m-1$.

The cases $m=1,2$ of \eqref{fo:Leibniz-1} and
\eqref{fo:derivation-1} follow from Steps~2 and~3.
Fix $p\ge2$ and assume inductively that the asserted formulas
and convergence statements hold up to order $p$.
Then
\begin{eqnarray*}
\Delta K_{\varphi_t}^{(p)}(z,w)
&=&
-\sum \frac{(p-1)!}{m_1!m_2!m_3!}
\int_\Omega
B_{m_3+1}(-\varphi_t)
K_{\varphi_t}^{(m_1)}(\cdot,w)
\overline{K_{\varphi_t}^{(m_2)}(\cdot,z)}
e^{-\varphi_t}\\
&&+
\sum \frac{(p-1)!}{m_1!m_2!m_3!}
\int_\Omega
B_{m_3+1}(-\varphi_{t_0})
K_{\varphi_{t_0}}^{(m_1)}(\cdot,w)
\overline{K_{\varphi_{t_0}}^{(m_2)}(\cdot,z)}
e^{-\varphi_{t_0}}.
\end{eqnarray*}
Set
\[
g_{t,k}:=-\partial_t^k(e^{-\varphi_t})
=-B_k(-\varphi_t)e^{-\varphi_t},
\ \ \  k\ge1.
\]
Using
\[
K_{\varphi_t}^{(m)}(z,w)
=
K_{\varphi_{t_0}}^{(m)}(z,w)
+\Delta K_{\varphi_t}^{(m)}(z,w),
\ \ \  0\le m\le p,
\]
we obtain
\begin{eqnarray}\label{fo:difference}
\Delta K_{\varphi_t}^{(p)}(z,w)
&=&
\sum \frac{(p-1)!}{m_1!m_2!m_3!}
\int_\Omega
K_{\varphi_{t_0}}^{(m_1)}(\cdot,w)
\overline{\Delta K_{\varphi_t}^{(m_2)}(\cdot,z)}
g_{t_0,m_3+1}
\nonumber\\
&&+
\sum \frac{(p-1)!}{m_1!m_2!m_3!}
\int_\Omega
\Delta K_{\varphi_t}^{(m_1)}(\cdot,w)
\overline{K_{\varphi_{t_0}}^{(m_2)}(\cdot,z)}
g_{t_0,m_3+1}
\nonumber\\
&&+
\sum \frac{(p-1)!}{m_1!m_2!m_3!}
\int_\Omega
K_{\varphi_{t_0}}^{(m_1)}(\cdot,w)
\overline{K_{\varphi_{t_0}}^{(m_2)}(\cdot,z)}
\Delta g_{t,m_3+1}
\nonumber\\
&&+\int_\Omega \mathcal R(z,w,\cdot,t),
\end{eqnarray}
where the sums range over $m_1+m_2+m_3=p-1$ with
$m_1,m_2,m_3\ge0$.

The regularity of $\varphi_t$ gives
\[
\left\|
\frac{\Delta g_{t,m_3+1}}{\Delta t}
-g_{t_0,m_3+2}
\right\|_{L^\infty(\Omega)}
\longrightarrow0
\ \ \ (t\to t_0).
\]
For $1\le q\le p$, put
\[
\mathcal E(\zeta,w,t;q)
:=
\frac{\Delta K_{\varphi_t}^{(q-1)}(\zeta,w)}{\Delta t}
-K_{\varphi_{t_0}}^{(q)}(\zeta,w).
\]
To prove \eqref{fo:estimate-1} for $m=p+1$, it is enough
to show that
\[
\|\mathcal E(\cdot,w,t;p)\|_{\varphi_{t_0}}
\longrightarrow0,
\ \ \ 
\|\overline{\mathcal E(\cdot,z,t;p)}\|_{\varphi_{t_0}}
\longrightarrow0.
\]

To justify the convergence in the weighted norm, observe
that orthogonality gives
\[
(P_{\varphi_t}-P_{\varphi_{t_0}})f
=
P_{\varphi_t}\left(
(1-e^{\varphi_t-\varphi_{t_0}})
(f-P_{\varphi_{t_0}}f)
\right),
\ \ \  f\in L^2(\Omega,\varphi_{t_0}).
\]
The uniform equivalence of the weighted norms therefore
yields
\[
\|(P_{\varphi_t}-P_{\varphi_{t_0}})f\|_{\varphi_{t_0}}
\le C|\Delta t|\|f\|_{\varphi_{t_0}}.
\]
Using the induction hypothesis \eqref{fo:derivation-1},
the boundedness and $L^\infty$ smoothness of the Bell
multipliers, and \eqref{eq:integ-1}, we obtain
\begin{align*}
&\|K_{\varphi_t}^{(j)}(\cdot,w)
-K_{\varphi_{t_0}}^{(j)}(\cdot,w)\|_{\varphi_{t_0}}\\
&\quad\le
C_j\left(
\|K_{\varphi_t}(\cdot,w)
-K_{\varphi_{t_0}}(\cdot,w)\|_{\varphi_{t_0}}
+|\Delta t|
\right)
\longrightarrow0,
\ \ \  0\le j\le p.
\end{align*}
In particular, applying the same
estimate at each nearby base point and using the uniform equivalence
of the weighted norms shows that the map
\[
s\longmapsto K_{\varphi_s}^{(p)}(\cdot,w)
\]
is continuous from a neighborhood of $t_0$ into $L^2(\Omega,\varphi_{t_0})$. It is therefore
Bochner integrable on the segment joining $t_0$ to $t$. The induction
hypothesis gives
$\partial_sK_{\varphi_s}^{(p-1)}(\zeta,w)
=K_{\varphi_s}^{(p)}(\zeta,w)$ pointwise in $\zeta$.
The scalar fundamental theorem of calculus together with the
Bochner--Fubini theorem therefore gives the following identity in
$L^2(\Omega,\varphi_{t_0})$, and justifies Minkowski's integral inequality in this fixed
weighted space.
For $\Delta t=t-t_0$, the fundamental theorem of calculus gives
\[
\mathcal E(\zeta,w,t;p)
=
\int_0^1
\left(
K_{\varphi_{t_0+\theta \Delta t}}^{(p)}(\zeta,w)
-K_{\varphi_{t_0}}^{(p)}(\zeta,w)
\right)\,d\theta.
\]
Hence Minkowski's inequality implies
\[
\|\mathcal E(\cdot,w,t;p)\|_{\varphi_{t_0}}
\le
\sup_{|t_1-t_0|\le|\Delta t|}
\|K_{\varphi_{t_1}}^{(p)}(\cdot,w)
-K_{\varphi_{t_0}}^{(p)}(\cdot,w)\|_{\varphi_{t_0}}
\longrightarrow0.
\]
The same argument gives
\[
\|\Delta K_{\varphi_t}^{(j)}(\cdot,w)\|_{\varphi_{t_0}}
=O(|\Delta t|),
\ \ \  0\le j\le p-1.
\]
In each summand of \eqref{fo:difference} one has
$m_1+m_2+m_3=p-1$, and hence $m_1,m_2\leq p-1$.
Therefore every kernel increment occurring in the remainder
$\mathcal R$ is covered by the preceding
$O(|\Delta t|)$ estimate. Moreover,
$\|\Delta g_{t,k}\|_{L^\infty(\Omega)}=O(|\Delta t|)$.
By construction, each term in $\mathcal R$ contains at least two
of these increments; all its remaining kernel factors are uniformly
bounded in $L^2(\Omega,\varphi_{t_0})$, and all remaining multiplier
factors are uniformly bounded in $L^\infty(\Omega)$. Consequently,
Cauchy--Schwarz gives
\[
\left|\int_\Omega \mathcal R(z,w,\zeta,t)\,dV(\zeta)\right|
=O(|\Delta t|^2).
\]
The conjugate estimate is identical.
Hence \eqref{fo:estimate-1} holds for $m=p+1$.
Dividing \eqref{fo:difference} by $\Delta t$ and letting
$t\to t_0$ then proves \eqref{fo:Leibniz-1} for $m=p+1$.

It remains to prove \eqref{fo:derivation-1} at order $p+1$.
Choose $\Omega_k\Subset\Omega$ increasing to $\Omega$, and
let $\chi_k$ be its characteristic function. Define
\[
I(z,w;k)
:=
-\int_\Omega
\chi_k(\cdot)
K_{\varphi_t}(\cdot,w)
\overline{K_{\varphi_t}(\cdot,z)}
(-\varphi_t)'
e^{-\varphi_t}.
\]
We use the conventions
\[
\Lambda_0=\{\varnothing\},\ \ \ 
\varnothing!=1,\ \ \ 
P_{\varphi_t,\varnothing}=\mathrm{Id}.
\]
Thus the induction hypothesis also covers the case of zero
derivatives.

By Leibniz's formula and the induction hypothesis,
\begin{eqnarray*}
\partial_t^p I(z,w;k)
&=&
-\sum_{\substack{m_1+m_2+m_3=p\\m_1,m_2,m_3\ge0}}
\frac{p!}{m_1!m_2!m_3!}
\int_\Omega
\chi_k
K_{\varphi_t}^{(m_1)}(\cdot,w)
\overline{K_{\varphi_t}^{(m_2)}(\cdot,z)}
B_{m_3+1}(-\varphi_t)e^{-\varphi_t}\\
&=&
-\sum
\frac{p!\,m_3}{\alpha!\beta!m_3!}
\int_\Omega
\chi_k
P_{\varphi_t,\alpha}
\bigl(K_{\varphi_t}(\cdot,w)\bigr)
\overline{
P_{\varphi_t,\beta}
\bigl(K_{\varphi_t}(\cdot,z)\bigr)}
B_{m_3}(-\varphi_t)e^{-\varphi_t}\\
&=&
\sum
\frac{p!\,m_3}{\alpha!\beta!m_3!}
P_{\varphi_t,\beta^-}
\circ P_{\varphi_t,m_3}
\left(
\chi_k
P_{\varphi_t,\alpha}
\bigl(K_{\varphi_t}(\cdot,w)\bigr)
\right)(z),
\end{eqnarray*}
where in the last two sums
\[
m_1,m_2\ge0,\ \ \  m_3\ge1,\ \ \ 
m_1+m_2+m_3=p+1,
\]
and
\[
\alpha=(\alpha_1,\ldots,\alpha_{r_1})\in\Lambda_{m_1},
\ \ \ 
\beta=(\beta_1,\ldots,\beta_{r_2})\in\Lambda_{m_2}.
\]
Here
\[
\beta^-:=(\beta_{r_2},\ldots,\beta_1),
\]
and when $m_1=0$ or $m_2=0$, the corresponding empty
composition is omitted.
The last equality follows from
Lemma~\ref{le:commutation-1}.

Since $\varphi_t$ is smooth up to the boundary,
\eqref{inq:smooth-estimate} yields
\begin{eqnarray*}
&&
\biggl\|
P_{\varphi_t,\beta^-}\circ P_{\varphi_t,m_3}
\left(
\chi_k
P_{\varphi_t,\alpha}
(K_{\varphi_t}(\cdot,w))
\right)\\
&&\hspace{25mm}
-
P_{\varphi_t,\beta^-}\circ P_{\varphi_t,m_3}
\circ P_{\varphi_t,\alpha}
(K_{\varphi_t}(\cdot,w))
\biggr\|_{\varphi_t}\\
&\le&
C\left\|
(\chi_k-1)
P_{\varphi_t,\alpha}
(K_{\varphi_t}(\cdot,w))
\right\|_{\varphi_t}
\longrightarrow0
\end{eqnarray*}
as $k\to\infty$.
On the left, \eqref{fo:Leibniz-1}, the preceding norm
estimates, and boundedness of the Bell multipliers give
\[
\partial_t^p I(z,w;k)
\longrightarrow K_{\varphi_t}^{(p+1)}(z,w).
\]
Indeed, each omitted integral is bounded by a constant times
\[
\|(1-\chi_k)K_{\varphi_t}^{(m_1)}(\cdot,w)\|_{\varphi_t}
\,
\|K_{\varphi_t}^{(m_2)}(\cdot,z)\|_{\varphi_t}.
\]
Thus
\begin{equation}\label{eq:F-alpha}
K_{\varphi_t}^{(p+1)}(z,w)
=
\sum
\frac{p!\,m_3}{\alpha!\beta!m_3!}
P_{\varphi_t,\beta^-}
\circ P_{\varphi_t,m_3}
\circ P_{\varphi_t,\alpha}
\bigl(K_{\varphi_t}(\cdot,w)\bigr)(z).
\end{equation}

We now collect the terms in \eqref{eq:F-alpha}
corresponding to the same ordered composition. Fix
\[
\gamma=(\gamma_1,\ldots,\gamma_r)\in\Lambda_{p+1,r}.
\]
For $j=1,\ldots,r$, take
\[
m_3=\gamma_j,\ \ \ 
\alpha=(\gamma_1,\ldots,\gamma_{j-1}),\ \ \ 
\beta=(\gamma_r,\ldots,\gamma_{j+1}),
\]
with the convention that an empty ordered composition is omitted.
Then
\[
P_{\varphi_t,\beta^-}
\circ P_{\varphi_t,m_3}
\circ P_{\varphi_t,\alpha}
=
P_{\varphi_t,\gamma},
\]
and the corresponding coefficient is
\[
\frac{p!\,\gamma_j}{\gamma!}.
\]
Conversely, every indexed term in \eqref{eq:F-alpha}
is obtained by choosing such a distinguished position $j$.
Therefore the total coefficient of
$P_{\varphi_t,\gamma}(K_{\varphi_t}(\cdot,w))$ is
\[
\sum_{j=1}^r\frac{p!\,\gamma_j}{\gamma!}
=
\frac{p!|\gamma|}{\gamma!}
=
\frac{(p+1)!}{\gamma!},
\]
since $|\gamma|=p+1$.
Consequently,
\[
K_{\varphi_t}^{(p+1)}(z,w)
=
\sum_{\gamma\in\Lambda_{p+1}}
\frac{(p+1)!}{\gamma!}
P_{\varphi_t,\gamma}
\bigl(K_{\varphi_t}(\cdot,w)\bigr)(z).
\]
The right-hand side is continuous in the norm
$\|\cdot\|_{\varphi_{t_0}}$ by the same argument used above.
Thus $K_{\varphi_t}^{(p+1)}(\cdot,w)$ is also continuous in this norm.
This proves \eqref{fo:derivation-1} for $m=p+1$
and completes the induction.

\section{Real analyticity}\label{sec:analyticity}

Throughout this section, $\Omega$ is bounded and pseudoconvex with
$C^2$ boundary. We use Theorem~\ref{th:Projection} to prove the
analyticity assertion of Theorem~\ref{th:Gevrey-1}.
We first prove the following weighted-norm statement.
It gives both the local uniform analyticity estimates and
the norm control used in Section~\ref{sec:logarithmic}.
Here $\tau\in\mathbb C$ denotes the complexification of the real
parameter $t$. A holomorphic extension is denoted by $K_\tau$ and
agrees with the original family on the real axis:
\[
\left.K_\tau(\cdot,w)\right|_{\tau=t}=K_t(\cdot,w).
\]
The reference parameter $t_0$ and the exponents defining the weighted
Hilbert spaces remain real; $\tau_0$ denotes a complex expansion point.

\begin{proposition}\label{prop:norm-analyticity}
For every $t_0>-1$ there exist $\eta_1,R>0$, with
$t_0-\eta_1>-1$, such that, for every $w\in\Omega$,
the map $t\mapsto K_t(\cdot,w)$, initially defined for real $t>-1$,
has a holomorphic extension $\tau\mapsto K_\tau(\cdot,w)$ to
$\{\tau\in\mathbb C:|\tau-t_0|<R\}$ with values in
$A^2_{t_0-\eta_1}(\Omega)$.
For every $S\Subset\Omega$,
\begin{equation}\label{eq:norm-analyticity}
\sup_{\substack{|t-t_0|\le R/2\\t\in\mathbb R,\ w\in S}}
 \|\partial_t^mK_t(\cdot,w)\|_{t_0-\eta_1}
 \le C_S(2/R)^m m!,\ \ \  m\ge0.
\end{equation}
\end{proposition}

By a dilation, we may assume $0<\delta<1$ on $\Omega$.
Indeed, for $c>0$,
\[
\delta_{c\Omega}(cz)=c\delta_\Omega(z),\ \ \ 
K_{c\Omega,t}(cz,cw)=c^{-2n-t}K_{\Omega,t}(z,w).
\]
The corresponding change of variables identifies the fixed
weighted spaces, and the factor $c^{-2n-\tau}$ is entire in $\tau$.
Fix $t_0>-1$ and choose $\eta_0,C_0$ as in
Theorem~\ref{th:Projection}, with $t_0-2\eta_0>-1$ and $C_0\ge1$.

\subsection{Extension to nearby weighted spaces}

\begin{lemma}\label{le:ra-extension}
The operator $P_{t_0}$ has consistent bounded extensions to
$L^2_{t_0+\sigma}(\Omega)$, $|\sigma|\le\eta_0$, satisfying
\begin{equation}\label{ra:two-sided}
\|P_{t_0}f\|_{t_0+\sigma}\le C_0\|f\|_{t_0+\sigma}.
\end{equation}
On each such space,
\[
P_{t_0}^2=P_{t_0},\ \ \ 
\operatorname{Ran}P_{t_0}=A^2_{t_0+\sigma}(\Omega).
\]
\end{lemma}

\begin{proof}
For real exponents $t_1,t_2$, the inclusions and the
duality pairing are
\begin{align*}
t_1\le t_2&\ \Longrightarrow\
 L^2_{t_1}(\Omega)\subset L^2_{t_2}(\Omega),\ \ \  \|f\|_{t_2}\le\|f\|_{t_1},
\\
|\langle f,g\rangle_{t_0}|
 &\le\|f\|_{t_0+\sigma}\|g\|_{t_0-\sigma}.
\end{align*}
Moreover,
\[
\|f\|_{t_0+\sigma}
 =\sup_{\substack{g\in C_c^\infty(\Omega)\\
                         \|g\|_{t_0-\sigma}\le1}}
       |\langle f,g\rangle_{t_0}|.
\]
Indeed, the maps
$f\mapsto\delta^{(t_0+\sigma)/2}f$ and
$g\mapsto\delta^{(t_0-\sigma)/2}g$ reduce
the preceding duality identities to $L^2(\Omega,dV)$
duality. The space $C_c^\infty(\Omega)$ is dense in each weighted
space: first truncate to compact subsets, and then mollify there,
where $\delta^{t_1}$ is bounded above and below by positive constants.

For $\sigma\le0$, \eqref{ra:two-sided} follows from
Theorem~\ref{th:Projection} with $t=t_0$ and $\eta=-\sigma$.
For $0<\sigma\le\eta_0$ and $f,g\in C_c^\infty(\Omega)$,
\begin{align*}
|\langle P_{t_0}f,g\rangle_{t_0}|
 =|\langle f,P_{t_0}g\rangle_{t_0}|
 \le\|f\|_{t_0+\sigma}\|P_{t_0}g\|_{t_0-\sigma}
 \le C_0\|f\|_{t_0+\sigma}\|g\|_{t_0-\sigma}.
\end{align*}
Taking the supremum over $g$ in the dual norm formula proves
\eqref{ra:two-sided} on $C_c^\infty(\Omega)$, hence on its completion.
If $f\in L^2_{t_1}(\Omega)\subset L^2_{t_2}(\Omega)$ and $f_j\in C_c^\infty(\Omega)$ converges to
$f$ in $L^2_{t_1}(\Omega)$, then
\[
\|f_j-f\|_{t_2}\le\|f_j-f\|_{t_1}\longrightarrow0.
\]
The two extensions are therefore the same limit in $L^2_{t_2}(\Omega)$.

For completeness, holomorphic subspaces are closed in these norms.
If $h_j\in\mathcal O(\Omega)$ and $h_j\to h$ in $L^2_{t_1}(\Omega)$, then,
for $E\Subset E'\Subset\Omega$, the mean-value inequality gives
\[
\sup_E|h_j-h_k|
 \le C_E\|h_j-h_k\|_{L^2(E')}
 \le C_{E,t_1}\|h_j-h_k\|_{t_1}\longrightarrow0.
\]
Hence $h$ has a holomorphic representative.
Approximating $f$ by compactly supported smooth functions now shows
\[
\operatorname{Ran}P_{t_0}\subset A^2_{t_0+\sigma}(\Omega).
\]

If $\sigma\le0$ and $h\in A^2_{t_0+\sigma}(\Omega)$, then $h\in A^2_{t_0}(\Omega)$ and
$P_{t_0}h=h$. If $0<\sigma\le\eta_0$, choose
$f_j\in C_c^\infty(\Omega)$ with $\|f_j-h\|_{t_0+\sigma}\to0$.
Theorem~\ref{th:Projection}, with $t=t_0+\sigma$ and $\eta=\sigma$,
implies
\[
P_{t_0+\sigma}f_j\in A^2_{t_0}(\Omega),\ \ \ 
\|P_{t_0+\sigma}f_j\|_{t_0}\le C_0\|f_j\|_{t_0}.
\]
Contractivity on $L^2_{t_0+\sigma}(\Omega)$ also gives
\[
\|P_{t_0+\sigma}f_j-h\|_{t_0+\sigma}
 \le\|f_j-h\|_{t_0+\sigma}\longrightarrow0.
\]
Consequently,
\begin{align*}
\|P_{t_0}h-h\|_{t_0+\sigma}
&\le
 \|P_{t_0}(h-P_{t_0+\sigma}f_j)\|_{t_0+\sigma}
 +\|P_{t_0+\sigma}f_j-h\|_{t_0+\sigma}\\
&\le(C_0+1)\|P_{t_0+\sigma}f_j-h\|_{t_0+\sigma}
 \longrightarrow0.
\end{align*}
Thus $P_{t_0}h=h$ on $A^2_{t_0+\sigma}(\Omega)$, which proves the range identity and idempotence.
\end{proof}

\begin{lemma}\label{le:ra-conjugation}
For every real $\sigma$ with $|\sigma|<\eta_0$, the two families
\[
\delta^{\pm(\tau-t_0)/2}P_{t_0}\delta^{\mp(\tau-t_0)/2}
\]
are bounded and holomorphic on $L^2_{t_0+\sigma}(\Omega)$ in the strip
$|\operatorname{Re}\tau-t_0|<\eta_0-|\sigma|$, and
\begin{equation}\label{ra:bound}
\left\|\delta^{\pm(\tau-t_0)/2}P_{t_0}
                    \delta^{\mp(\tau-t_0)/2}\right\|_{\mathcal L(L^2_{t_0+\sigma}(\Omega))}
 \le C_0.
\end{equation}
On $L^2_{t_0}(\Omega)$,
\begin{equation}\label{ra:adjoint}
\left(\delta^{(\tau-t_0)/2}P_{t_0}\delta^{(t_0-\tau)/2}\right)^*
 =\delta^{(t_0-\bar \tau)/2}P_{t_0}\delta^{(\bar \tau-t_0)/2}.
\end{equation}
\end{lemma}

\begin{proof}
Complex powers mean $\delta^\tau=\exp(\tau\log\delta)$.
For each choice of sign,
\begin{align*}
\|\delta^{\mp(\tau-t_0)/2}u\|_{t_0+\sigma\pm(\operatorname{Re}\tau-t_0)}
 &=\|u\|_{t_0+\sigma},\\
\|\delta^{\pm(\tau-t_0)/2}v\|_{t_0+\sigma}
 &=\|v\|_{t_0+\sigma\pm(\operatorname{Re}\tau-t_0)}.
\end{align*}
The exponent shift satisfies
\[
|\sigma\pm(\operatorname{Re}\tau-t_0)|<\eta_0.
\]
Lemma~\ref{le:ra-extension} therefore gives
\begin{align*}
&\|\delta^{\pm(\tau-t_0)/2}P_{t_0}
                         (\delta^{\mp(\tau-t_0)/2}u)\|_{t_0+\sigma}\\
&\quad=
 \|P_{t_0}(\delta^{\mp(\tau-t_0)/2}u)\|
                         _{t_0+\sigma\pm(\operatorname{Re}\tau-t_0)}\\
&\quad\le C_0
 \|\delta^{\mp(\tau-t_0)/2}u\|
                         _{t_0+\sigma\pm(\operatorname{Re}\tau-t_0)}
 =C_0\|u\|_{t_0+\sigma}.
\end{align*}
This also specifies the spaces on which all factors are composed.

For $f,g\in C_c^\infty(\Omega)$,
\begin{eqnarray}\label{ra:matrix-elements}
\left\langle
 \delta^{\pm(\tau-t_0)/2}P_{t_0}\delta^{\mp(\tau-t_0)/2}f,g
 \right\rangle_{t_0+\sigma}
 &=&
 \iint_{\Omega\times\Omega}K_{t_0}(\zeta,\xi)
 e^{\pm\frac{\tau-t_0}{2}(\log\delta(\zeta)-\log\delta(\xi))}
 f(\xi)\overline{g(\zeta)}\notag\\
&&\times\delta(\xi)^{t_0}\delta(\zeta)^{t_0+\sigma}\,dV(\xi)\,dV(\zeta).
\end{eqnarray}
On $\operatorname{supp}g\times\operatorname{supp}f$,
$K_{t_0}$ and both logarithms are bounded. For each compact
parameter set and each $m\ge0$, the absolute value of the $m$th
parameter derivative of the integrand is bounded by
\[
2^{-m}\,C
 |\log\delta(\zeta)-\log\delta(\xi)|^m|f(\xi)g(\zeta)|.
\]
Thus \eqref{ra:matrix-elements} is entire.
The bound \eqref{ra:bound} and
Lemma~\ref{le:ra-weak-holomorphy} prove operator-norm holomorphy.

For the adjoint, first take $f,g\in C_c^\infty(\Omega)$ and use the
selfadjointness of $P_{t_0}$ on $L^2_{t_0}(\Omega)$:
\begin{align*}
\left\langle
 \delta^{(\tau-t_0)/2}P_{t_0}\delta^{(t_0-\tau)/2}f,g
 \right\rangle_{t_0}=\left\langle
 P_{t_0}(\delta^{(t_0-\tau)/2}f),
             \delta^{(\bar \tau-t_0)/2}g\right\rangle_{t_0}
=\left\langle
 f,\delta^{(t_0-\bar \tau)/2}P_{t_0}
                 (\delta^{(\bar \tau-t_0)/2}g)\right\rangle_{t_0}.
\end{align*}
Both sides extend continuously to all $f,g\in L^2_{t_0}(\Omega)$.
This proves \eqref{ra:adjoint}.
\end{proof}

\subsection{Recovery of the weighted projection}

For real $t$ with $|t-t_0|<\eta_0$, multiplication gives a
surjective isometry
\[
L^2_t(\Omega)\longrightarrow L^2_{t_0}(\Omega),\ \ \ 
f\longmapsto\delta^{(t-t_0)/2}f.
\]
Lemma~\ref{le:ra-extension} shows that
\[
\delta^{(t-t_0)/2}P_{t_0}\delta^{(t_0-t)/2}
\]
is a bounded idempotent with range $\delta^{(t-t_0)/2}A^2_t(\Omega)$.
The orthogonal projection onto the same range is
\[
\delta^{(t-t_0)/2}P_t\delta^{(t_0-t)/2}.
\]
Indeed, for $f\in L^2_{t_0}(\Omega)$ and $h\in A^2_t(\Omega)$,
\begin{align*}
\left\langle
 f-\delta^{(t-t_0)/2}P_t(\delta^{(t_0-t)/2}f),
             \delta^{(t-t_0)/2}h\right\rangle_{t_0}
=
 \left\langle
 \delta^{(t_0-t)/2}f-P_t(\delta^{(t_0-t)/2}f),h
 \right\rangle_t=0.
\end{align*}
Lemmas~\ref{le:ra-conjugation}
and~\ref{le:ra-projection-identity} now give
\begin{equation}\label{ra:projection-identity}
\begin{aligned}
\delta^{(t-t_0)/2}P_t\delta^{(t_0-t)/2}
=
 \delta^{(t-t_0)/2}P_{t_0}\delta^{(t_0-t)/2}
 \cdot
 \left[
 \delta^{(t-t_0)/2}P_{t_0}\delta^{(t_0-t)/2}
 +\delta^{(t_0-t)/2}P_{t_0}\delta^{(t-t_0)/2}-\mathrm{Id}
 \right]^{-1}.
\end{aligned}
\end{equation}

Replace $t$ by $\tau$ on the right-hand side of
\eqref{ra:projection-identity}. We construct the inverse
of the operator in square brackets simultaneously on
$L^2_{t_0}(\Omega)$ and $L^2_{t_0-\eta_0/2}(\Omega)$.
For $\sigma\in\{0,-\eta_0/2\}$,
Lemma~\ref{le:ra-conjugation} shows that
\[
\tau\longmapsto
\delta^{\pm(\tau-t_0)/2}P_{t_0}\delta^{\mp(\tau-t_0)/2}
\]
is holomorphic on a neighborhood of
$\{|\tau-t_0|\le\eta_0/4\}$, with operator norm at most $C_0$
on $L^2_{t_0+\sigma}(\Omega)$.
Its value at $\tau=t_0$ is $P_{t_0}$.
Applying \eqref{ra:operator-coefficients} on the circle
$|\tau-t_0|=\eta_0/4$, we obtain
\begin{equation}\label{ra:operator-difference}
\left\|
\delta^{\pm(\tau-t_0)/2}P_{t_0}\delta^{\mp(\tau-t_0)/2}
-P_{t_0}
\right\|_{\mathcal L(L^2_{t_0+\sigma}(\Omega))}
\le
C_0\frac{|\tau-t_0|}{\eta_0/4-|\tau-t_0|},
\ \ \  |\tau-t_0|<\eta_0/4.
\end{equation}
At $\tau=t_0$, the corresponding bracket is
$2P_{t_0}-\mathrm{Id}$, and
\[
(2P_{t_0}-\mathrm{Id})^2=\mathrm{Id},\ \ \ 
\|2P_{t_0}-\mathrm{Id}\|_{\mathcal L(L^2_{t_0+\sigma}(\Omega))}\le2C_0+1.
\]
Choose, for example,
\begin{equation}\label{ra:radius-choice}
R=\frac{\eta_0}{64C_0(2C_0+1)}.
\end{equation}
For $|\tau-t_0|\le2R$, \eqref{ra:operator-difference} gives
\begin{align}
&\left\|
 \left[
 \delta^{(\tau-t_0)/2}P_{t_0}\delta^{(t_0-\tau)/2}
 +\delta^{(t_0-\tau)/2}P_{t_0}\delta^{(\tau-t_0)/2}-2P_{t_0}
 \right](2P_{t_0}-\mathrm{Id})
 \right\|_{\mathcal L(L^2_{t_0+\sigma}(\Omega))}\notag\\
&\quad\le
 \frac{4C_0(2C_0+1)R}{\eta_0/4-2R}<\frac12.
\label{ra:inverse-smallness}
\end{align}
Lemma~\ref{le:ra-inversion} proves that the inverse in
\eqref{ra:projection-identity} is holomorphic on both spaces,
with norm at most $2(2C_0+1)$.
Thus its product with the first factor has norm at most
$2C_0(2C_0+1)$.

By Lemma~\ref{le:ra-extension}, the conjugated factors agree on
$C_c^\infty(\Omega)$ in both weighted spaces, hence on
$L^2_{t_0-\eta_0/2}(\Omega)$ by density and boundedness.
The operators in square brackets therefore agree on that space.
For $u\in L^2_{t_0-\eta_0/2}(\Omega)$, the inverse computed there solves
the same equation in $L^2_{t_0}(\Omega)$. Uniqueness of the inverse on
$L^2_{t_0}(\Omega)$ shows that the two solutions coincide.
Thus the right-hand side of \eqref{ra:projection-identity}
is consistent on the two spaces $L^2_{t_0}(\Omega)$ and $L^2_{t_0-\eta_0/2}(\Omega)$.

For $|\tau-t_0|<2R$, the whole expression
$\delta^{(\tau-t_0)/2}P_\tau\delta^{(t_0-\tau)/2}$
will denote the holomorphic right-hand side of
\eqref{ra:projection-identity}, with $t$ replaced by $\tau$.
For $\tau=t\in\mathbb R$, it agrees with the original expression.
This convention does not define an orthogonal Bergman projection
for a complex-valued weight.

\subsection{Holomorphy in a fixed weighted space}

\begin{proof}[Proof of Proposition~\ref{prop:norm-analyticity}]
Fix $S\Subset\Omega$. Choose
\[
0<2\varepsilon_0<\operatorname{dist}(S,\partial\Omega),\ \ \ 
f\in C_c^\infty(B(0,\varepsilon_0)),\ \ \ 
f\ge0,\ \ \  \int f\,dV=1,
\]
where $f$ is real and radial. The mean-value property gives
\[
\int_\Omega h(\zeta)f(\zeta-w)\,dV(\zeta)=h(w)
\]
for every holomorphic or antiholomorphic $h$ and for $w$ near $S$.
Consequently, for real $t$,
\begin{equation}\label{ra:kernel-test}
P_t\bigl(\delta^{-t}f(\cdot-w)\bigr)(z)
=\int_\Omega K_t(z,\zeta)f(\zeta-w)\,dV(\zeta)
=K_t(z,w).
\end{equation}
For $|\tau-t_0|<2R$, define
\begin{equation}\label{eq:strong-kernel-identity}
K_\tau(\cdot,w)
:=\delta^{(t_0-\tau)/2}
\bigl(\delta^{(\tau-t_0)/2}P_\tau\delta^{(t_0-\tau)/2}\bigr)
\bigl(\delta^{-(\tau+t_0)/2}f(\cdot-w)\bigr).
\end{equation}
For $\tau=t\in\mathbb R$, this agrees with $K_t(\cdot,w)$
by \eqref{ra:kernel-test}. We use the consistent continuation of
\eqref{ra:projection-identity} on $L^2_{t_0-\eta_0/2}$.
The middle factor in \eqref{eq:strong-kernel-identity} is
holomorphic for $|\tau-t_0|<2R$ and has norm at most
$2C_0(2C_0+1)$ there. The last vector is entire in this space,
since $\log\delta$ is bounded on its support.

Set $\eta_1=\eta_0/4$. For the first factor,
\begin{align*}
\|\delta^{(t_0-\tau)/2}u\|_{t_0-\eta_1}^2
=\int_\Omega |u|^2\delta^{t_0-\eta_0/2}
 \delta^{\eta_1-(\operatorname{Re}\tau-t_0)}\,dV
\le\|u\|_{t_0-\eta_0/2}^2,
\ \ \  |\operatorname{Re}\tau-t_0|<\eta_1.
\end{align*}
For $\eta>0$ and $k\in\mathbb N_0$, we shall use
\begin{equation}\label{eq:elementary}
|\log\delta|^k\le k!\eta^{-k}\delta^{-\eta}.
\end{equation}
Fix $\tau_0$ in the preceding strip and put
$\eta_2=(\eta_1-\operatorname{Re}\tau_0+t_0)/2>0$.
By \eqref{eq:elementary},
\[
\left\|
\frac{1}{j!}\left(-\frac{\log\delta}{2}\right)^j
\delta^{(t_0-\tau_0)/2}u
\right\|_{t_0-\eta_1}
\le (2\eta_2)^{-j}\|u\|_{t_0-\eta_0/2},
\ \ \  j\ge0.
\]
Therefore, for $\tau=\tau_0+\Delta\tau$,
\[
\delta^{(t_0-\tau_0-\Delta\tau)/2}u
=\sum_{j=0}^\infty\frac{(\Delta\tau)^j}{j!}
 \left(-\frac{\log\delta}{2}\right)^j
 \delta^{(t_0-\tau_0)/2}u,
\ \ \  |\Delta\tau|<2\eta_2,
\]
with convergence in the operator norm from
$L^2_{t_0-\eta_0/2}(\Omega)$ to $L^2_{t_0-\eta_1}(\Omega)$.
This proves the required multiplier holomorphy.

Since $2R<\eta_1$, the right-hand side of
\eqref{eq:strong-kernel-identity} is holomorphic for
$|\tau-t_0|<2R$ with values in $L^2_{t_0-\eta_1}(\Omega)$, and
\begin{equation}\label{ra:strong-kernel-bound}
\sup_{\substack{|\tau-t_0|\le R\\w\in S}}
\|K_\tau(\cdot,w)\|_{t_0-\eta_1}
\le
2C_0(2C_0+1)
\sup_{\substack{|\tau-t_0|\le R\\w\in S}}
\|\delta^{-(\tau+t_0)/2}f(\cdot-w)\|_{t_0-\eta_0/2}
\le C_S.
\end{equation}
For $v\in(A^2_{t_0-\eta_1}(\Omega))^\perp$, the function
$\langle K_\tau(\cdot,w),v\rangle_{t_0-\eta_1}$ is holomorphic
and vanishes when $\tau=t\in\mathbb R$. The identity theorem implies
\[
K_\tau(\cdot,w)\in
(A^2_{t_0-\eta_1}(\Omega))^{\perp\perp}
=A^2_{t_0-\eta_1}(\Omega),\ \ \  |\tau-t_0|<2R.
\]
Here closedness follows from Lemma~\ref{le:ra-extension}.
Different choices of $f$ give the same continuation by the
identity theorem, since they agree when $\tau=t\in\mathbb R$.

For $t\in\mathbb R$ with $|t-t_0|\le R/2$, the Cauchy formula in
$A^2_{t_0-\eta_1}(\Omega)$ gives
\[
K_t^{(m)}(\cdot,w)
=\left.\partial_\tau^m K_\tau(\cdot,w)\right|_{\tau=t}
=\frac{m!}{2\pi i}\int_{|\tau-t|=R/2}
 \frac{K_\tau(\cdot,w)}{(\tau-t)^{m+1}}\,d\tau,
\ \ \  K_t^{(m)}:=\partial_t^mK_t.
\]
Together with \eqref{ra:strong-kernel-bound}, this proves
\eqref{eq:norm-analyticity}.
\end{proof}

\subsection{The kernel and its spatial derivatives}

\begin{proof}[Proof of the analyticity assertion in
Theorem~\ref{th:Gevrey-1}]
We use the continuation in
\eqref{eq:strong-kernel-identity} and the same constants
$\eta_1,R$. Let $\partial_z^\mu,\partial_w^\nu$ denote
derivatives in the real spatial coordinates.
Differentiation with respect to the translation parameter $w$
gives, in $L^2_{t_0-\eta_1}(\Omega)$,
\[
\partial_w^\nu K_\tau(\cdot,w)
=\delta^{(t_0-\tau)/2}
\bigl(\delta^{(\tau-t_0)/2}P_\tau\delta^{(t_0-\tau)/2}\bigr)
\bigl(\delta^{-(\tau+t_0)/2}\partial_w^\nu f(\cdot-w)\bigr).
\]
These functions are holomorphic in $\tau$, take values in
$A^2_{t_0-\eta_1}(\Omega)$, and satisfy
\[
\sup_{\substack{|\tau-t_0|\le R\\w\in S}}
\|\partial_w^\nu K_\tau(\cdot,w)\|_{t_0-\eta_1}
\le C_{S,\nu}.
\]
No derivative of $\delta$ occurs. The local mean-value and
Cauchy estimates give
\[
\sup_{z\in S}|\partial_z^\mu h(z)|
\le C_{S,\mu}\|h\|_{t_0-\eta_1},
\ \ \  h\in A^2_{t_0-\eta_1}(\Omega).
\]
Consequently,
\[
\sup_{\substack{|\tau-t_0|\le R\\z,w\in S}}
|\partial_z^\mu\partial_w^\nu K_\tau(z,w)|
\le C_{S,\mu,\nu}.
\]
Applying the Cauchy formula on $|\tau-t|=R/2$, we obtain
\begin{equation}\label{ra:local-factorial}
\sup_{\substack{t\in\mathbb R\\|t-t_0|\le R/2}}
\|K_t^{(m)}\|_{C^\ell(S\times S)}
\le C_{S,\ell}(2/R)^m m!,\ \ \  m\ge0.
\end{equation}
The same estimates give convergence of the Taylor series in
every $C^\ell(S\times S)$ norm.

Finally, cover a compact interval $J\Subset(-1,\infty)$ by
finitely many intervals
\[
J\subset\bigcup_{j=1}^N(t_j-R_j/2,t_j+R_j/2),
\ \ \  R_J=\frac12\min_{1\le j\le N}R_j>0.
\]
By \eqref{ra:local-factorial},
\[
\sup_{t\in J}\|K_t^{(m)}\|_{C^\ell(S\times S)}
\le\max_{1\le j\le N}C_{j,S,\ell}\,R_J^{-m}m!.
\]
This proves \eqref{eq:analytic-main}. By
\eqref{ra:radius-choice}, $R_J$ is independent of $S$ and $\ell$.
\end{proof}

\section{Derivative formulas for logarithmic boundary-distance weights}
\label{sec:logarithmic}

Throughout this section, $\Omega$ is bounded and pseudoconvex
with $C^2$ boundary. By a dilation, we may assume $0<\delta<1$.
We use the weighted-norm holomorphy established in
Proposition~\ref{prop:norm-analyticity} and follow the
Leibniz and composition argument of Section~3.

\subsection{Logarithmic multipliers}

Since $\varphi_t=-t\log\delta$, we have
\[
B_k(-\varphi_t)=(\log\delta)^k,\ \ \  k\ge1.
\]
Thus, for $\alpha=(\alpha_1,\ldots,\alpha_r)\in\Lambda_m$,
\[
P_{t,k}f:=P_t\bigl(-f(\log\delta)^k\bigr),
\ \ \ 
P_{t,\alpha}
:=P_{t,\alpha_r}\circ\cdots\circ P_{t,\alpha_1}.
\]
We also use the conventions
\[
\Lambda_0=\{\varnothing\},\ \ \ 
\varnothing!=1,\ \ \ 
P_{t,\varnothing}=\mathrm{Id}.
\]

Fix $t_0>-1$ and $S\Subset\Omega$. We use the constants
$\eta_0,C_0,\eta_1,R$ chosen in Section~\ref{sec:analyticity}.
The proof of Theorem~\ref{th:Openness}, with this choice of
$\eta_0$ from Theorem~\ref{th:Projection}, gives
\[
\sup_{\substack{|t-t_0|\le\eta_0/2\\w\in S}}
\|K_t(\cdot,w)\|_{t-\eta_0}^2\le\widetilde C_0.
\]

\begin{lemma}\label{le:Openness}
For $|t-t_0|\le\eta_0/2$, $0\le\sigma'<\sigma\le\eta_0$,
$f\in L^2_{t-\sigma}(\Omega)$, and
$\alpha\in\Lambda_{m,r}$,
\begin{equation}\label{eq:word-mapping}
\|P_{t,\alpha}f\|_{t-\sigma'}
\le C_0^{r/2}\sqrt{(2\alpha)!}
 \left(\frac{r}{\sigma-\sigma'}\right)^m
 \|f\|_{t-\sigma},
\end{equation}
where $(2\alpha)!:=\prod_{j=1}^r(2\alpha_j)!$.
In particular, for $w\in S$,
\begin{equation}\label{inq:gevrey-estimate-1}
\|P_{t,\alpha}(K_t(\cdot,w))\|_{t-\eta_0/2}^2
\le\widetilde C_0(2\alpha)!(mR_1)^{2m},
\ \ \  R_1:=2\eta_0^{-1}\sqrt{C_0}.
\end{equation}
All these compositions are well defined.
\end{lemma}

\begin{proof}
For $0<\eta\le\sigma\le\eta_0$, Theorem~\ref{th:Projection}
and \eqref{eq:elementary} give
\begin{align*}
\|P_{t,k}f\|_{t-\sigma+\eta}^2
&\le C_0\int_\Omega
 |f|^2|\log\delta|^{2k}\delta^{t-\sigma+\eta}\,dV\\
&\le C_0(2k)!\eta^{-2k}\|f\|_{t-\sigma}^2.
\end{align*}
The multiplier belongs to the weighted space on which the
projection is applied. Take $\eta=(\sigma-\sigma')/r$ and
apply this estimate successively with
$k=\alpha_1,\ldots,\alpha_r$, using the exponents
\[
t-\sigma,\ t-\sigma+\eta,\ \ldots,\ t-\sigma+r\eta=t-\sigma'.
\]
Multiplication of the $r$ bounds proves \eqref{eq:word-mapping}.
With $\sigma=\eta_0$ and $\sigma'=\eta_0/2$, it follows that
\begin{align*}
\|P_{t,\alpha}(K_t(\cdot,w))\|_{t-\eta_0/2}^2
&\le C_0^r(2\alpha)!
 \left(\frac{2r}{\eta_0}\right)^{2m}
 \|K_t(\cdot,w)\|_{t-\eta_0}^2\\
&\le\widetilde C_0(2\alpha)!(mR_1)^{2m},
\end{align*}
since $r\le m$ and $C_0\ge1$.
\end{proof}

We also use the following analogue of
Lemma~\ref{le:commutation-1}.
\begin{lemma}\label{le:commutation}
Let $\chi_V$ be the characteristic function of a measurable
set $V\Subset\Omega$, and let $f\in A^2_t(\Omega)$.
For $\alpha\in\Lambda_m$, $\gamma\in\mathbb N_+$, and
$|t-t_0|\le\eta_0/2$,
\[
\int_\Omega\chi_V f\,
\overline{P_{t,\alpha}(K_t(\cdot,z))}
(\log\delta)^\gamma\delta^t\,dV
=-P_{t,\alpha^-}P_{t,\gamma}(\chi_V f)(z),
\]
where $\alpha^-=(\alpha_r,\ldots,\alpha_1)$.
The identity also holds for the empty composition.
\end{lemma}

\begin{proof}
If $u,v\in A^2_t(\Omega)$ and
$(\log\delta)^ju,(\log\delta)^jv\in L^2_t(\Omega)$, then
\[
\langle P_{t,j}u,v\rangle_t
=-\langle(\log\delta)^ju,v\rangle_t
=\langle u,P_{t,j}v\rangle_t.
\]
The function $\chi_V f$ belongs to $L^2_{t-\eta_0}(\Omega)$.
By \eqref{eq:word-mapping} and Theorem~\ref{th:Openness},
all partial compositions below belong to $L^2_{t-\eta_0/2}(\Omega)$.
Moreover, \eqref{eq:elementary} gives
\[
\|(\log\delta)^j u\|_t^2
\le(2j)!(\eta_0/2)^{-2j}\|u\|_{t-\eta_0/2}^2.
\]
Thus the preceding pairing identity applies at each transfer,
and orthogonality gives
\begin{align*}
&\left\langle\chi_V f(\log\delta)^\gamma,
 P_{t,\alpha}(K_t(\cdot,z))\right\rangle_t\\
&\quad=-\left\langle P_{t,\gamma}(\chi_V f),
 P_{t,\alpha}(K_t(\cdot,z))\right\rangle_t\\
&\quad=-\left\langle P_{t,\alpha^-}P_{t,\gamma}(\chi_V f),
 K_t(\cdot,z)\right\rangle_t\\
&\quad=-P_{t,\alpha^-}P_{t,\gamma}(\chi_V f)(z).
\end{align*}
\end{proof}

\subsection{The higher derivative formula}

\begin{theorem}\label{th:formula}
Let $\Omega$ be a bounded pseudoconvex domain with $C^2$ boundary.
For $t>-1$, $z,w\in\Omega$, and $m\ge1$,
\begin{equation}\label{fo:derivation}
K_t^{(m)}(z,w)
=\sum_{\alpha\in\Lambda_m}\frac{m!}{\alpha!}
 P_{t,\alpha}\bigl(K_t(\cdot,w)\bigr)(z).
\end{equation}
\end{theorem}

\begin{proof}
Write $K_t^{(j)}=\partial_t^jK_t$.
By Proposition~\ref{prop:norm-analyticity}, the holomorphic
extension $\tau\mapsto K_\tau(\cdot,w)$ gives
\[
K_t^{(j)}(\cdot,w)
=\left.\partial_\tau^jK_\tau(\cdot,w)\right|_{\tau=t}
\in A^2_{t_0-\eta_1}(\Omega),\ \ \  j\ge0,
\]
for $t\in\mathbb R$ with $|t-t_0|<R$. We restrict the real parameter to
$|t-t_0|<R/2$.
For $f,g\in L^2_{t_0-\eta_1}(\Omega)$ and $q\in\mathbb N_0$,
\begin{align}
\int_\Omega |fg|\,|\log\delta|^q\delta^t\,dV
&\le\|\delta^{t-t_0+\eta_1}|\log\delta|^q\|_{L^\infty}
 \|f\|_{t_0-\eta_1}\|g\|_{t_0-\eta_1}\notag\\
&\le q!(2/\eta_1)^q
 \|f\|_{t_0-\eta_1}\|g\|_{t_0-\eta_1}.
\label{eq:log-bilinear}
\end{align}
Here $R<\eta_1$ and \eqref{eq:elementary} are used.
The same estimate, with $q$ replaced by $q+j$, applies to
\[
\partial_t^j\bigl(\delta^{t-t_0+\eta_1}(\log\delta)^q\bigr)
=\delta^{t-t_0+\eta_1}(\log\delta)^{q+j}.
\]
Taylor's theorem and the preceding bounds show that
$t\mapsto\delta^{t-t_0+\eta_1}(\log\delta)^q=:M_q(t)$ is $C^\infty$
with values in $L^\infty(\Omega)$ for $|t-t_0|<R/2$.
Proposition~\ref{prop:norm-analyticity} gives smoothness of the
kernel sections in the fixed space $L^2_{t_0-\eta_1}(\Omega)$.
The estimate
\[
\left|\int_\Omega f\overline g M_q(t)\,
                   \delta^{t_0-\eta_1}\,dV\right|
\le\|M_q(t)\|_{L^\infty}\|f\|_{t_0-\eta_1}\|g\|_{t_0-\eta_1}
\]
therefore justifies repeated Leibniz differentiation in all
the integrals below.
The reproducing property gives
\begin{equation}\label{eq:key}
K_t(z,w)-K_{t_0}(z,w)
=\int_\Omega K_t(\zeta,w)\overline{K_{t_0}(\zeta,z)}
 \bigl(\delta(\zeta)^{t_0}-\delta(\zeta)^t\bigr)\,dV(\zeta).
\end{equation}
Indeed, both kernel sections belong to
$A^2_{t_0-\eta_1}(\Omega)\subset A^2_{t_0}(\Omega)\cap A^2_t(\Omega)$.
Differentiating \eqref{eq:key} at $t=t_0$, and then using
the arbitrariness of $t_0$, gives
\begin{align}
K_t^{(1)}(z,w)
&=-\int_\Omega K_t(\zeta,w)\overline{K_t(\zeta,z)}
 (\log\delta(\zeta))\delta(\zeta)^t\,dV(\zeta)\notag\\
&=P_{t,1}\bigl(K_t(\cdot,w)\bigr)(z).
\label{eq:C^1}
\end{align}
Differentiating this identity $m-1$ times yields
\begin{align}
K_t^{(m)}(z,w)
=-\!\!\sum_{\substack{m_1+m_2+m_3=m-1\\m_1,m_2,m_3\ge0}}
&\frac{(m-1)!}{m_1!m_2!m_3!}
\int_\Omega K_t^{(m_1)}(\zeta,w)
 \overline{K_t^{(m_2)}(\zeta,z)}\notag\\
&{}\cdot(\log\delta(\zeta))^{m_3+1}
 \delta(\zeta)^t\,dV(\zeta).
\label{fo:Leibniz}
\end{align}

It remains to prove \eqref{fo:derivation} by induction.
The case $m=1$ is \eqref{eq:C^1}.
Fix $p\ge1$ and assume the formula for all $m\le p$,
including $m=0$ with the empty composition.
Choose $\Omega_k\Subset\Omega$ increasing to $\Omega$,
and let $\chi_k$ be its characteristic function. Define
\[
I(z,w;k):=-\int_\Omega
\chi_k K_t(\cdot,w)\overline{K_t(\cdot,z)}
(\log\delta)\delta^t\,dV.
\]
By Leibniz's formula and the induction hypothesis,
\begin{align*}
\partial_t^p I(z,w;k)
=-\!\!\sum_{\substack{m_1+m_2+m_3=p\\m_1,m_2,m_3\ge0}}
&\frac{p!}{m_1!m_2!m_3!}
\int_\Omega\chi_k K_t^{(m_1)}(\cdot,w)
 \overline{K_t^{(m_2)}(\cdot,z)}
 (\log\delta)^{m_3+1}\delta^t\,dV\\
=-\sum&\frac{p!\,m_3}{\alpha!\beta!m_3!}
\int_\Omega\chi_k P_{t,\alpha}(K_t(\cdot,w))
 \overline{P_{t,\beta}(K_t(\cdot,z))}
 (\log\delta)^{m_3}\delta^t\,dV\\
=\sum&\frac{p!\,m_3}{\alpha!\beta!m_3!}
P_{t,\beta^-}P_{t,m_3}
 \bigl(\chi_kP_{t,\alpha}(K_t(\cdot,w))\bigr)(z).
\end{align*}
In the last two sums,
\[
m_1,m_2\ge0,\ \ \  m_3\ge1,\ \ \ 
m_1+m_2+m_3=p+1,\ \ \ 
\alpha\in\Lambda_{m_1},\quad\beta\in\Lambda_{m_2}.
\]
The last equality follows from Lemma~\ref{le:commutation}.
The coefficient is
\[
\frac{p!}{m_1!m_2!(m_3-1)!}
\frac{m_1!}{\alpha!}\frac{m_2!}{\beta!}
=\frac{p!\,m_3}{\alpha!\beta!m_3!}.
\]

We next remove the cutoffs. By \eqref{fo:Leibniz} and
\eqref{eq:log-bilinear},
\begin{align*}
&|\partial_t^p I(z,w;k)-K_t^{(p+1)}(z,w)|\\
&\quad\le C_p\!\!\sum_{\substack{m_1+m_2+m_3=p\\m_1,m_2,m_3\ge0}}
\|(1-\chi_k)K_t^{(m_1)}(\cdot,w)\|_{t_0-\eta_1}
\|K_t^{(m_2)}(\cdot,z)\|_{t_0-\eta_1}
\longrightarrow0.
\end{align*}
On the other hand, \eqref{eq:word-mapping}, with
$\sigma=\eta_0/2$ and $\sigma'=\eta_0/4$, gives
\begin{align*}
&\left\|P_{t,\beta^-}P_{t,m_3}
 \bigl((\chi_k-1)P_{t,\alpha}(K_t(\cdot,w))\bigr)
 \right\|_{t-\eta_0/4}\\
&\quad\le C_{p,\alpha,\beta,m_3}
 \|(\chi_k-1)P_{t,\alpha}(K_t(\cdot,w))\|_{t-\eta_0/2}
\longrightarrow0.
\end{align*}

For the first cutoff limit above, each
$K_t^{(m_1)}(\cdot,w)$ belongs to the fixed space $L^2_{t_0-\eta_1}(\Omega)$;
since $|1-\chi_k|\le1$ and $\chi_k\to1$ pointwise, dominated
convergence gives
$\|(1-\chi_k)K_t^{(m_1)}(\cdot,w)\|_{t_0-\eta_1}\to0$.
Equation~\eqref{eq:log-bilinear} then controls the corresponding
integrals. For the second limit, Lemma~\ref{le:Openness} gives
$P_{t,\alpha}(K_t(\cdot,w))\in L^2_{t-\eta_0/2}(\Omega)$ when
$\alpha\ne\varnothing$; for $\alpha=\varnothing$ the same
membership follows from Theorem~\ref{th:Openness}.
Dominated convergence again makes the right-hand norm above tend
to zero. The bounded operator estimate \eqref{eq:word-mapping}
places the resulting compositions in $A^2_{t-\eta_0/4}(\Omega)$, where
point evaluation at $z$ is continuous. We may therefore pass to
the limit in each term of the finite sum and obtain
\begin{equation}\label{eq:log-F-alpha}
K_t^{(p+1)}(z,w)
=\sum\frac{p!\,m_3}{\alpha!\beta!m_3!}
P_{t,\beta^-}P_{t,m_3}P_{t,\alpha}
 \bigl(K_t(\cdot,w)\bigr)(z).
\end{equation}

We collect the terms in \eqref{eq:log-F-alpha} corresponding
to the same ordered composition. Fix
$\gamma=(\gamma_1,\ldots,\gamma_r)\in\Lambda_{p+1,r}$.
For $j=1,\ldots,r$, take
\[
m_3=\gamma_j,\ \ \ 
\alpha=(\gamma_1,\ldots,\gamma_{j-1}),\ \ \ 
\beta=(\gamma_r,\ldots,\gamma_{j+1}),
\]
with the convention that an empty multi-index is omitted.
Then
\[
P_{t,\beta^-}P_{t,m_3}P_{t,\alpha}=P_{t,\gamma},
\ \ \  \alpha!\beta!m_3!=\gamma!,
\]
and the corresponding coefficient is $p!\gamma_j/\gamma!$.
Conversely, every indexed term in \eqref{eq:log-F-alpha}
is obtained by choosing such a distinguished position $j$.
The total coefficient is therefore
\[
\sum_{j=1}^r\frac{p!\gamma_j}{\gamma!}
=\frac{p!|\gamma|}{\gamma!}
=\frac{(p+1)!}{\gamma!}.
\]
Consequently,
\[
K_t^{(p+1)}(z,w)
=\sum_{\gamma\in\Lambda_{p+1}}
\frac{(p+1)!}{\gamma!}
P_{t,\gamma}\bigl(K_t(\cdot,w)\bigr)(z).
\]
This proves \eqref{fo:derivation} for $m=p+1$ and completes
the induction.
\end{proof}

\section*{Acknowledgements}
The second author thanks Professor Bo-Yong Chen for helpful
discussions and many valuable suggestions.

\section*{Declaration of generative AI and AI-assisted technologies
in the manuscript preparation process}

During the preparation of this work, the authors used ChatGPT (OpenAI) 
to edit the text and develop the argument of complexification, 
which establishes real analyticity in the parameter. 
This improves upon our previously obtained Gevrey-2 regularity result. 
The authors reviewed and edited all AI-generated content, 
and take full responsibility
for the content of this article.

\end{document}